\documentclass[12pt, a4paper, reqno]{amsart}
\usepackage[margin=3cm]{geometry}

\usepackage{latexsym,amsmath,amsthm,amssymb,mathrsfs,esint,ulem}
\usepackage{graphics}
\usepackage{enumerate}
\usepackage{hyperref}
\usepackage{epsfig}

\usepackage{xcolor}
\numberwithin{equation}{section}

\newcommand{\e}{\varepsilon}
\newcommand{\wto}{\rightharpoonup}
\newcommand{\R}{\mathbb{R}}

\newcommand{\Z}{\mathbb{Z}}

\newcommand{\ol}{\overline}
\renewcommand\O{{\Omega}}
\renewcommand{\H}{\mathcal{H}^2}

\newcommand{\mthree}{\mathbb{M}^{3{\times}3}}
\newcommand{\msym}{\mathbb{M}^{3{\times}3}_{\rm sym}}
\newcommand{\id}{{\rm \bf{id}}}

\newcommand\p{\partial}
\newcommand\vece{{\vec{e}_r}}\newcommand\ds{\displaystyle}
\newcommand\I{\mathscr I}
\newcommand\J{\mathscr J}
\newcommand\G{\mathscr G}
\newcommand\F{\mathscr F}
\newcommand\E{\mathscr E}
\newcommand\bE{\mathbf E}
\newcommand\X{{\mathscr X}}
\newcommand\Os{{\O\setminus \ol B_a}}
\newcommand\A{\mathcal A}
\newcommand\Oe{\O^\e}

\newcommand\ue{u^\e}
\newcommand\hue{{\hat u}^\e}
\newcommand\hwe{{\hat w}^\e}

\newcommand\Bie{B^i_{\e a}}
\newcommand\bBie{{\ol B}^i_{\e a}}
\newcommand\lfl {\lambda_{f\!\ell}}

\newcommand{\Dir}{\Gamma}
\newcommand{\Aa}{\mathbb A}
\newcommand{\Aah}{\mathbb A_{\rm hom}}

\newcommand{\chiF}{\lambda_F}
\newcommand{\chiij}{\lambda_{ij}}
\newcommand{\chiFij}{\lambda_{F_{ij}}}
\newcommand{\chijk}{\lambda_{jk}}

\newcommand\Ie{I_\e}
\newcommand\ba{\begin{array}}
\newcommand\ea{\end{array}}
\renewcommand\oe{\omega^\e}

\newcommand\K{\tilde\Omega^\e}

\renewcommand\phi{\varphi}

\theoremstyle{plain}
\begingroup
\newtheorem{theorem}{Theorem}[section]
\newtheorem{lemma}[theorem]{Lemma}
\newtheorem{proposition}[theorem]{Proposition}

\endgroup

\theoremstyle{definition}
\begingroup

\newtheorem{remark}[theorem]{Remark}

\endgroup

\DeclareMathOperator{\dist}{dist}

\DeclareMathOperator{\cof}{cof}
\DeclareMathOperator{\dive}{div}

\DeclareMathOperator{\tr}{tr}

\def \into {\int_\Omega}

\def \ep {\varepsilon}

\def \Om {\Omega}

\def \ph {\varphi}

\def \RR {\mathbb R}

\def \A {\mathcal{A}}

\def \F {\mathcal{F}}
\def \G {\mathcal{G}}

\def \beq {\begin{equation}}
\def \eeq {\end{equation}}
\def \ba {\begin{array}}
\def \ea {\end{array}}

\begin{document}
\title[Liquid filled perforated elastomers]{{Liquid Filled Elastomers}: From Linearization to Elastic Enhancement}

\author[J. Casado D\'iaz] {Juan Casado D\'iaz}
\address[J. Casado D\'iaz]{ Departamento de Ecuaciones Diferenciales y An\'alisis Num\'erico, Universidad de Sevilla, Campus de Reina Mercedes,
41012 Sevilla, Spain}
\email[J. Casado D\'iaz]{jcasadod@us.es}

\author[G.A. Francfort] {Gilles A. Francfort}
\address[G.A. Francfort]{Flatiron Institute, 162 Fifth Avenue, New York, NY10010, USA}
\email[G. A. Francfort]{gfrancfort@flatironinstitute.org}
\author[O. Lopez-Pamies] {Oscar Lopez-Pamies}
\address[O. Lopez-Pamies]{Department of Civil and Environmental Engineering,  University of Illinois at Urbana--Champaign, IL 61801, USA}
\email[O. Lopez-Pamies]{pamies@illinois.edu}
\author[M.G. Mora] {Maria Giovanna Mora}
\address[M.G. Mora]{Dipartimento di Matematica, Universit\`a di Pavia, Via Ferrata 5, 27100 Pavia, Italy}
\email{mariagiovanna.mora@unipv.it}

\begin{abstract}
Surface tension at  cavity walls can play havoc with the mechanical properties of perforated soft solids when the cavities are filled with a fluid. This study is an investigation of the macroscopic elastic properties of elastomers embedding spherical cavities filled with a pressurized liquid in the presence of surface tension, starting with the linearization of the fully nonlinear model and ending with the enhancement properties of the linearized model when many such liquid filled cavities are present.
\end{abstract}

\maketitle

\section{Introduction}

The study of the mechanics  of interfaces in the continuum has a long and rich history with origins dating back to the classical works of Young \cite{Young1805} and Laplace \cite{Laplace1806} on interfaces between fluids in the early 1800's  and of Gibbs \cite{Gibbs1928} on the more general case of interfaces between solids and fluids in the 1870's. Yet it was only in 1975 that complete descriptions of the kinematics, the concept of interfacial stress, and the balance of linear and angular momenta of bodies containing interfaces were properly formulated, even when specialized to the basic  case of elastic interfaces \cite{Gurtin75a,Gurtin75b}.    The results remained abstract at the time, most certainly because of the technical difficulties in measuring and tailoring  mechanical and physical properties of interfaces. In the early 2000's,  the onset of new synthesis and characterization tools reinvigorated the study of interfaces in soft matter.

In this context, elastomers filled with liquid --- as opposed to solid --- inclusions are a  recent trend in the soft matter community because they  exhibit remarkable mechanical and physical properties; see e.g. \cite{Syleetal15,LDLP17,Yunetal19}. In particular, the interfacial physics in these soft material systems can be actively tailored to  enhance or impede deformability. While the addition of liquid inclusions should increase the macroscopic deformability of the material,  the behavior of the solid/liquid interfaces, if negligible when the inclusions are ``large'', may counteract this increase and lead to stiffening when the inclusions become sufficiently ``small''.

As a first step in our understanding of this  paradigm, a recent contribution by one of us \cite{GLP} derives the governing equations that describe the mechanical response of a hyperelastic solid filled with initially spherical inclusions made of a pressurized hyperelastic fluid when the solid/fluid interface is hyperelastic and possesses an initial surface tension. Arguably, this corresponds to the most basic type of elastomer filled with liquid inclusions.

From a mechanics standpoint, the main objectives of this work are twofold. First, we derive the linearization of the governing equations put forth in \cite{GLP} in the limit of small deformations. Second, within that linearized setting, we derive the homogenization limit of a periodic distribution of liquid inclusions as the period gets smaller. Formal derivations of both results were proposed in \cite{GLP,GLLP}. Our analysis corroborates those, although even an attentive reader may be at pains to check that the results are identical because of differential geometric intricacies.

From a mathematical standpoint, given an elastic energy density $W$ and an interaction surface term $\J$ on the boundary of a ball $B_a\subset \O$ that will be detailed below, we propose to linearize the energy
$$
\E_\e (y)=\int_\Os W(\nabla y)\, dx + \J(y) -\e \int_\Os f\cdot y\, dx
$$
when the external load (here $\e f$)  is indeed of order $\e$, $\e$ being a small parameter. This is by now a classical problem that was first  handled in \cite{DMNP} for  { finite elasticity} by computing the $\Gamma(L^2)$-limit of
$$
\I_\e(u):=\frac1{\e^2}\int_\O W(I+\e \nabla u)\; dx
$$
{for a standard elastic energy density $W$} and showing the $L^2$-compactness of almost minimizers of $v\mapsto\I_\e(v)-\int_\O f\cdot v\; dx.$
The celebrated rigidity result of \cite{FJM} plays a pivotal role in the analysis.

Since \cite{DMNP}, a similar linearization process has been implemented in a variety of settings generally assuming that the relevant forces were of order $\e$ and rescaling the energy accordingly as above. In that spirit, the work which is closest to the current investigation is \cite{MR} where live pressure loads are applied to an elastic body.

The current setting introduces a new feature in the analysis, namely a pre-stress due to the liquid inclusions. This in turn changes the order of the various contributions and results in a breakdown of the $\Gamma$-convergence process. We were unfortunately unable to complete that process without the addition of a vanishing higher order contribution to the energy. This is because, barring the presence of such an additional term, we cannot hope to get a bound on the $L^2$-norm of the tangential gradient of the (rescaled) field along the boundary --- that is, the solid/liquid interface --- of the liquid cavity (a sphere). Consequently, we are adding an appropriately vanishing second order term to our energy (see \eqref{eq.reg}, \eqref{eta_e}). A similar technique was recently used in \cite{ADMLP} to handle the linearization of a multi-well elastic energy.

Our first  result is the linearization Theorem~\ref{thm} which gives rise from a P.D.E. standpoint to a highly non trivial set of equations both on the solid part of the domain and on the boundary of the liquid inclusions,  that is, along the solid/liquid interface (see Remark \ref{rem.PDE}).

Our second result, Theorem~\ref{thm.hom}, is a periodic homogenization statement on the linearized system with an appropriate rescaling of the surface tension on each inclusion. The resulting homogenized behavior ends up being purely elastic.  However, the expression for the homogenized Hooke's law incorporates  a memory of the presence of surface tension on the solid/liquid interfaces. This is achieved through a somewhat intricate unfolding of the oscillating fields which heavily draws on the specific spectral properties of spherical harmonics and, to our knowledge is the first time a homogenization process is performed on a system that couples bulk and surface P.D.E.'s.

The homogenization result  promotes elastic enhancement  as detailed in Subsection \ref{sec.enhance}. Technical hurdles prevent the derivation of a general enhancement result so we have to illustrate its occurrence on a uniaxial strain and for an isotropic base material; see Proposition \ref{prop.enhance}. We are confident that enhancement
can always be achieved for large enough surface tensions, although currently defeated in our attempts to prove such a general statement.

\medskip

Now for a few mathematical prerequisites. We recall that, for any $C^1$-manifold $M$ embedded in an open set $\O\subset\R^3$ and any smooth field  $u:\O\to \R^3$, the tangential gradient $\nabla_\tau u$ of $u$ at $x\in M$ is defined as
$$
\nabla_\tau u(x)= \nabla u(x)-\nabla u(x)\nu(x)\otimes\nu(x)
$$
where $\nu(x)$ is the unit normal to $M$ at $x$. Similarly, the tangential divergence of  $u$ at $x\in M$  is defined as
$$
\dive_\tau u(x)=\dive u(x)-(\nabla u)^T(x)\nu(x)\cdot\nu(x).
$$

Note that,  for any smooth vector field $v$ on a smooth oriented manifold $M$ with normal unit vector $\nu$,
\begin{equation}\label{eq.tau-nabla-div}
\nabla_\tau(\dive_\tau v)= (I-\nu\otimes\nu)\dive_\tau (\nabla_\tau v)^T-\nabla_\tau \nu(\nabla_\tau v)^T\nu,
\end{equation}
where the tangential divergence of a tensor-valued function $S$ is defined as $\dive_\tau S\cdot e=\dive_\tau(S^Te)$ for any vector $e$
(see \cite[Lemma~2.6]{CRMSP}).

We also recall a few useful algebraic identities. In the following $\mthree$ stands for the space of $(3\times 3)$-matrices with $I$ as the identity matrix
and $\msym$ for the subspace of symmetric $(3\times 3)$-matrices.

For any $A,B\in\mthree$ with $\det A\neq0$ one has
\begin{equation}\label{eq.alg1}
\cof(A+B)=\cof A +\cof B +\frac1{\det A} \big( (\cof A\cdot B)\cof A - (\cof A) B^T(\cof A) \big)
\end{equation}
(see, e.g, \cite[Proposition~1.6]{PGVC}).

Also Cayley-Hamilton Theorem implies that, for any $A\in\mthree$,
$$\frac12\big( (\tr A)^2-\tr A^2\big)A-(\tr A)A^2+A^3=(\det A) I$$
so that
\begin{equation}\label{eq.alg2}
\cof A= \frac12\big( (\tr A)^2-\tr A^2\big)I-(\tr A)A^T+(A^T)^2.
\end{equation}
from which we also obtain that
\begin{equation}\label{eq.alg3}
\tr(\cof A)=\frac12\big((\tr A)^2- \tr A^2\big).
\end{equation}

Finally, if $\{\tau_1,\tau_2,\nu\}$ form an orthonormal basis of  vectors  with $\tau_1\times\tau_2=\nu$,
 \begin{equation}\label{eq.alg4}(\cof A)\nu=A\tau_1\times A\tau_2 \quad \mbox{ for every }A\in\mthree.
 \end{equation}

 Notationwise, we denote by $B_r$ the open ball of center $0$ and radius $r>0$ and by $\id$ the identity mapping   ($\id(x)=x$).
 Also, throughout $\{\vec{e}_i, i=1,2,3\}$ is the canonical orthonormal basis of $\R^3$ and $\vece$ is the unit  normal at the point under consideration on any sphere. 

We always omit the target space when writing a norm, so, for example, if $u:\O\to \R^3$, $\|u\|_{L^2(\O)}$ is the $L^2(\O;\R^3)$-norm of $u$.

We  say that a sequence $(a_\e)\subset\R$ indexed by $\e\searrow0$ is of the order of $\e^\alpha$ (and write $a_\e \simeq \e^\alpha$) if there exist two positive constants $c,c'$ such that $c\e^\alpha\le a_\e\le c'\e^\alpha$ for all $\e$. The symbol $\subset\!\subset$ means ``compactly contained in"
and  the symbol $\sharp$ stands for ``periodic". The characteristic function of a set $A$ is denoted by $\chi_A$.

Finally, ``$\cdot$" stands for the Euclidean inner product on $\R^3$ as well as for the Frobenius inner product on matrices, i.e. $A\cdot B:= {\rm tr }(B^TA)$ for $A,B\in \mthree$.

\section{The nonlinear formulation}

\subsection{Setting}
Let $\O\subset\R^3$ be a Lipschitz bounded domain containing a closed ball $\ol B_a$ with $a>0$.

The elastomer occupying  $\O\setminus \ol B_a$ is assumed to possess an elastic energy density
$W:\mthree\to[0,+\infty]$ that satisfies what are by now the usual assumptions of geometrically non-linear elasticity theory, namely,
\begin{itemize}
\item[(i)] $W(F)=W(RF)$ for every $F\in\mthree$ and $R\in SO(3)$,\smallskip
\item[(ii)] $W(I)=0$, $W(F)= +\infty$ if $\det F\le 0$,\smallskip
\item[(iii)] $W(F)\geq c\dist^2(F, SO(3))$ for every $F\in\mthree$, for some $c>0$,\smallskip
\item[(iv)] $W$ is $C^2$ in a neighborhood of $SO(3)$ and $\ds \Aa:={\p^2 W}/{\p F^2}(I)$, \smallskip
\item[(v)] $W(F)\to +\infty$ as $\det F\to 0^+$.\smallskip
\end{itemize}

\begin{remark}\label{rm.positdef}
The assumptions on $W$ imply that the quadratic form $F\mapsto Q(F):=\Aa F\cdot F$ is positive definite on symmetric matrices.
\hfill\P
\end{remark}

The ball $B_a$ is filled with a compressible pressurized liquid. The initial Cauchy pressure is $p$.
We take the internal energy density of the liquid to be
\begin{equation}\label{eq.en-fl} W_{f\!\ell}(F):= \frac\lfl 2 (\det F-1)^2-p\det F, \quad\lfl >0.\end{equation}
In other words, the liquid is assumed to be an elastic fluid with bulk modulus $\lfl$.
This corresponds to a first Piola-Kirchhoff stress of the form
\begin{equation}\label{eq.PK}\ds P=\frac{\p W_{f\!\ell}}{\p F}(F)= (-p+\lfl (\det F -1))\cof F,\end{equation}
 hence to a Cauchy stress of the form $$\Sigma=(-p +\lfl (\det F-1))I.$$

Because of the presence of interfacial forces at the solid/liquid interface $\p B_a$, the interface $\p B_a$ exerts a normal force per unit surface area $-\gamma \kappa(y(x))\nu_{y(x)}$ on the liquid. In this last expression, $\gamma\geq 0$ stands for the surface tension of the solid/liquid interface, an intrinsic property of the interface at hand, while $\kappa(y(x)), x\in \p B_a$, is the mean curvature at the image of $x$ under the deformation $y: \O\to\R^3$, and $\nu_{y(x)}$ is the exterior normal to $y(B_a)$ at $y(x)$.

We assume that, in the absence of any external loading process, the initial liquid pressure $p$ is  so that the hydrostatic pressure $-p\vece$  equilibrates the surface tension without deformation of the ball so that, $y\equiv\id$, $\kappa\equiv -2/a$ and
\begin{equation}\label{eq.eqm}
 p=2\gamma/a.
\end{equation}

The contribution to the energy of the surface tension is $\gamma \H(\p y(B_a))$. { Indeed,  assuming that $y$ is a $C^2$-diffeomorphism on $\ol \O$ so that, in particular, $y(\O)$ is open, and $\p y(B_a)$ is a $C^2$-manifold, and considering a deformation  of the form $\Phi^\e(y)=y+\e z(y)$ with $z$ smooth and compactly supported in $y(\O)$, we get (see \cite[Theorems 7.31, 7.34]{AFP})}
\begin{equation}\label{eq.der-area}
 \frac{\p}{\p \e}\H(\p y(B_a))|_{\e=0}=\int_{\p y(B_a)} \dive_\tau z \,d\H =- \int_{\p y(B_a)}\kappa(y(x))\nu_{y(x)}\cdot z\, d\H.
\end{equation}
Now, since $y$ is one to one and $\p y(B_a)=y(\p B_a)$, the area formula   yields
\begin{equation}\label{eq.area}
 \H(\p y(B_a))=\int_{\p B_a} |\cof \nabla y\vece|\, d\H
 \end{equation}
and that contribution becomes
 \begin{equation}\label{eq.st-cont}
 \gamma \int_{\p B_a} |\cof \nabla y\vece| \,d\H.
 \end{equation}

Further, the fluid must be in equilibrium; thus $\dive P=0$ in $B_a$ with $P$ defined by \eqref{eq.PK}.
Explicit computations yield $$0=\dive P=\lfl (\cof \nabla y)^T\nabla(\det\nabla y).$$
Since $\det\nabla y>0$, the matrix $(\cof \nabla y)^T$ is invertible, hence $\det \nabla y \equiv {\rm cst}$ in $B_a$ and the constant must be $|y(B_a)|/|B_a|$. So, recalling \eqref{eq.en-fl},   the contribution of the fluid to the energy is
\begin{equation}\label{eq.fl-cont}
-p |y(B_a)|+ \frac{\lfl }2 |B_a| \left(\frac{|y(B_a)|}{|B_a|}-1\right)^2.
\end{equation}

In view of \eqref{eq.st-cont}, \eqref{eq.fl-cont}, we conclude that, for
 $y$ a $C^2$-diffeomorphism on $\ol \O$ and in the presence of a body load applied to the solid part of $\O$ and of density $\e f$ with $\e>0$ and $f:\Os\to\R^3$, the total energy is given by
\begin{equation}\label{eq.tot-en}
\E_\e (y):=\int_\Os W(\nabla y)\, dx + \J(y) -\e \int_\Os f\cdot y\, dx
\end{equation}
where
\begin{equation}\label{eq.tot-en2}
\J(y):=
\gamma \int_{\p B_a} |(\cof\nabla y)\vece|\,d\H+ \frac{\lfl}2 |B_a| \left(\frac{|y(B_a)|}{|B_a|}-1\right)^2-p |y(B_a)|.
\end{equation}

Note that we can always write
$$
|y(B_a)|=\int_{B_a}\det\nabla y(x)\, dx=\frac13\int_{\partial B_a} \cof\nabla y \vece\cdot y\, d\H.
$$

\begin{remark}\label{rem.bl}
From a mechanical standpoint the imposition of a (dead) body load acting only on the solid part of the domain $\Os$ might seem unrealistic. It would certainly be more appropriate to impose surface loads on $\p\O\setminus\overline\Dir$, where $\Dir$ is the Dirichlet part of the boundary. Of course surface loads are generally live loads and this would add an additional level of complexity which we do not want to address here. The interested reader is directed to \cite{MR} where a study of linearization in the presence of pressure loads is undertaken.
\hfill\P\end{remark}

\subsection{Existence of a minimizer in the absence of external loads}

As seen in the previous Subsection, relation \eqref{eq.eqm} ensures that the pressure in the fluid $B_a$ is balanced by the surface tension on $\p B_a$ in the initial configuration. Consequently, in the absence of external loadings, no loads are  applied to the solid part of $\O$, that is to $\O\setminus\ol B_a$.  Thus the identity mapping $\id$ should in  essence be a ``stationary point" for $\J$, this independently of the values of $\gamma$ or $\lfl$. Of course, such will only be true for smooth variations because of the determinant constraint.

In the context of \eqref{eq.tot-en}, we investigate if and when the identity mapping $\id$ is an energetic minimizer when $f\equiv 0$.

Formally, the isoperimetric inequality implies that
$$
\H(\p y(B_a))\ge (36\pi)^{1/3} |y(B_a)|^{2/3}=:C|y(B_a)|^{2/3}
$$
so that, in view of \eqref{eq.area},
$$
\J(y)\ge  \gamma C|y(B_a)|^{2/3}+\frac{\lfl }2 |B_a| \left(\frac{|y(B_a)|}{|B_a|}-1\right)^2-p |y(B_a)|.
$$

Set
$$
\Phi(t):= \gamma Ct^{2/3}+\frac{\lfl }2 |B_a| \left(\frac{t}{|B_a|}-1\right)^2-p t\quad \text{ for } t\ge 0,
$$ and note that
$$
\Phi(0)= \frac23 \pi \lfl  a^3>0, \quad \Phi'(0)=+\infty, \quad \lim_{t\to+\infty} \Phi(t)=\lim_{t\to+\infty} \Phi'(t)=+\infty,
$$
and $\Phi''(t)<0$ for $t<t^*$, $\Phi''(t)>0$ for $t>t^*$ with
$$
t^*= \left(\frac{9\lfl }{2\gamma C |B_a|}\right)^{-3/4}.
$$
Now, recalling \eqref{eq.eqm},  $\Phi'(t^*)=((2^{11}/3^3)\lfl )^{1/4}(\gamma /a)^{3/4}-(\lfl +2\gamma /a)$.
The maximum in $\gamma/a$ of the previous expression is attained at $\gamma/ a= (3/2)\lfl $ and it is $0$, so that $\Phi'(t^*)<0$  except when $\gamma=(3/2)\lfl a$ in which case $\Phi'(t^*)=0$.

In view of the already established properties of $\Phi$, this means that, for $\gamma\ne (3/2)\lfl a$, $\Phi$ has a unique maximizer at some point $t'<t^*$ and a unique minimizer at some point $t''>t^*$. Now, $\Phi'(|B_a|)=0$ and $\Phi''(|B_a|)=a^{-3}/\pi(3\lfl /4-\gamma a^{-1}/2)$, which is positive if and only if $\gamma< (3/2)\lfl a$. In that case, $|B_a|$ must be a minimizer and it is unique.

For $\gamma=(3/2)\lfl a$ we have $\Phi'(t)\ge 0$ for $t>0$ and thus $\Phi$ is increasing. Further $\Phi(|B_a|)>\Phi(0)$, so  $|B_a|$ is not a minimizer of $\Phi$.

Since the elastic energy $\int_{\O\setminus\ol B_a}W(\nabla y) dx$ is strictly positive for $y\ne Rx+c$ we conclude  that, provided that $\gamma< (3/2)\lfl a$, $\E_\e$ is only minimized at $y=\id$, and also possibly at $y=Rx+c$ for $R\in SO(3)$ and $c\in\R^3$, if we have imposed no boundary conditions on $y$.

\medskip
In conclusion we have obtained the following
\begin{lemma}\label{lem.id} In the absence of external loadings and under assumption \eqref{eq.eqm}, $\id$ is the unique minimizer of $\E_\e$ (possibly up to rotations and translations)
provided that
\begin{equation}\label{eq.cond-min}
\gamma< \frac32\lfl a \quad(\mbox{or equivalently } p< 3\lfl ).
\end{equation}
Further $\E_\e(\id)=\ds\frac43\pi\gamma a^2$.
\end{lemma}

\begin{remark}\label{rm.incomliqu}
Condition (\ref{eq.cond-min}) is always satisfied by incompressible liquids, i.e. when $\lfl =+\infty$.\hfill\P
\end{remark}

From now onward, we  restrict the setting to that for which {\em both} \eqref{eq.eqm} and \eqref{eq.cond-min} hold true.

\section{Linearization in the presence of a higher order regularization}\label{sec.lin}
As mentioned in the introduction, the linearization process cannot succeed without the addition of a regularizing term.

For simplicity we will assume henceforth that the domain $\O$ is clamped on a part $\Dir$ of its boundary, so that we can apply Poincar\'e's inequality.
We assume $\Dir$ to be a non-empty subset of $\p\O$, open in the relative topology of $\p\O$ and such that ${\rm cap}_2(\ol\Dir\setminus\Dir)=0$ (we refer to \cite[Section 4.7]{EG} for the notion of $2$-capacity).
This last condition is needed to ensure a suitable density result in the proof of the $\Gamma$-limsup inequality.

Let $p>3$ and
$$
W^{2,p}_\Dir(\O;\R^3):=\{u\in W^{2,p}(\O;\R^3): \ u=0\mbox{ on }\Dir\}.
$$
For a given load $f\in L^2(\Os;\R^3)$ we consider the regularized functional $\F_\e$ defined on $C^2_{loc}({\O};\R^3)\cap W^{2,p}_\Dir(\O;\R^3)$ as
\begin{equation}\label{eq.reg}
\F_\e(u):=\frac1{\e^2}\Big(\E_\e(\id+\e u)-\E_\e(\id)\Big) +\eta_\e\int_\Os |\nabla^2 u|^p\, dx,
\end{equation}
where
\begin{equation}\label{eta_e}
\e^{p/3}\le\eta_\e\stackrel{\e}{\to}0.
\end{equation}
We define
\begin{equation}\label{def.A}
\A:=\{u\in H^1(\Os;\R^3): \ u\cdot\vece\in H^1(\p B_a) \text{ and }u=0 \text{ on }\Dir\}\end{equation}
 and
the functional on $\A$
\begin{multline*}
\F(u):=\frac12 \int_\Os Q(\bE u)\, dx
+\frac\gamma2\int_{\p B_a}|\nabla_\tau (u\cdot\vece)|^2\, d\H
-\frac{\gamma}{a^2}\int_{\p B_a}  |u\cdot\vece|^2 \, d\H
\\
+\frac{\lfl}{2|B_a|} \left(\int_{\p B_a}  u\cdot\vece\,d\H\right)^2
-\int_\Os f\cdot u\, dx
\end{multline*}
with $\bE u:=1/2(\nabla u+\nabla u^T)$ and $Q$ defined as in Remark \ref{rm.positdef}.

Our first main result is the following compactness and convergence theorem.
\begin{theorem}\label{thm}
Assume that \eqref{eq.eqm} and \eqref{eq.cond-min} hold true.
Let $(u^\e)$ be a sequence in $C^2_{loc}({\O};\R^3)\cap W^{2,p}_\Dir(\O;\R^3)$ such that
\begin{equation}\label{min-seq}
\F_\e(u^\e)\leq C.
\end{equation}
Then there exists $u\in\A$ such that, up to subsequences,
$u^\e$ converge to $u$  weakly in $H^1(\Os;\R^3)$ and $u^\e\cdot\vece\wto u\cdot\vece$ weakly in $H^1(\p B_a)$.

Further, the $\Gamma$-limit of $\F_\e$ in the strong $L^2(\O\setminus\ol B_a;\R^3)$ topology is precisely $\F$.
\end{theorem}

\begin{remark}
Condition \eqref{min-seq} is clearly satisfied by any minimizing sequence $(u^\e)$ of~$\F_\e$ since $\F_\e(0)=0$.\hfill\P
\end{remark}

\begin{remark}\label{rm.posit}
By \eqref{PW-sphere} and \eqref{eq.cond-min} we have
$$
\frac\gamma2\int_{\p B_a}|\nabla_\tau (u\cdot\vece)|^2\, d\H
-\frac{\gamma}{a^2}\int_{\p B_a}  |u\cdot\vece|^2 \, d\H
+\frac{\lfl}{2|B_a|} \left(\int_{\p B_a}  u\cdot\vece\,d\H\right)^2\ge 0
$$
for any $u\in \A$. Furthermore, again by \eqref{PW-sphere} and \eqref{eq.cond-min}, for every $\delta>0$ small enough the following coercivity property holds:
\begin{multline*}
\frac\gamma2\int_{\p B_a}|\nabla_\tau (u\cdot\vece)|^2\, d\H
-\frac{\gamma}{a^2}\int_{\p B_a}  |u\cdot\vece|^2 \, d\H
+\frac{\lfl}{2|B_a|} \left(\int_{\p B_a}  u\cdot\vece\,d\H\right)^2 \\
\ge \delta \int_{\p B_a}|\nabla_\tau (u\cdot\vece)|^2\, d\H
- 2\frac{\delta}{a^2} \int_{\p B_a}  |u\cdot\vece|^2 \, d\H
\end{multline*}
for any $u\in \A$.
\hfill\P\end{remark}

\begin{proof}[Proof of Theorem~\ref{thm}]
We split the proof into three steps.\medskip

\noindent {\sf Step~1: Compactness.}
By \eqref{min-seq} we deduce that
\begin{multline}
\frac1{\e^2} \left(\int_\Os W(I+\e\nabla u^\e)\; dx+ \J(\id+\e u^\e)-\frac43\pi\gamma a^2\right) +\eta_\e\int_\Os |\nabla^2 u^\e|^p\, dx  \\
\leq  \int_\Os f\cdot u^\e\, dx +C
 \leq  \|f\|_{L^2(\Os)}\|u^\e\|_{L^2(\Os)}+C. \label{bound-I-v}
\end{multline}
Since $\J(\id+\e u^\e)\geq \frac43\pi\gamma a^2$ by Lemma \ref{lem.id}, for some possibly larger $C>0$ we have
$$
\int_\Os W(I+\e\nabla u^\e)\, dx\leq C\e^2(\|u^\e\|_{L^2(\Os)}+1).
$$
We now apply the rigidity estimate \cite[Theorem~3.1]{FJM} and conclude  that there exists a constant $R^\e\in SO(3)$ such that
\begin{equation}\label{rigidity}
\|I+\e\nabla u^\e-R^\e\|^2_{L^2(\Os)}\le C\e^2(\|u^\e\|_{L^2(\Os)}+1).
\end{equation}
Let $\xi^\e$ be the mean of the function $\id+\e u^\e-R^\e x$ on $\Os$. By Poincar\'e-Wirtinger's inequality we obtain
$$
\|\id+\e u^\e-R^\e x-\xi^\e\|_{H^1(\Os)}^2\leq C\e^2(\|u^\e\|_{L^2(\Os)}+1).
$$
Since $u^\e=0$ on $\Dir$, the continuity of the trace operator yields
$$
\|\id-R^\e x-\xi^\e\|_{L^2(\Dir)}^2\leq C\|\id+\e u^\e-R^\e x-\xi^\e\|_{H^1(\Os)}^2\leq C\e^2(\|u^\e\|_{L^2(\Os)}+1).
$$
By \cite[Lemma~3.3]{DMNP} (see also the proof of \cite[Proposition~3.4]{DMNP}) this implies that
$$
|I-R^\e|\leq C\e^2(\|u^\e\|_{L^2(\Os)}+1).
$$
Consequently, by \eqref{rigidity} and the boundary condition on $\Dir$, we conclude that
\begin{equation}\label{eq.H1bound}
(u^\e) \text{ is bounded in } H^1(\Os;\R^3).
\end{equation}

Therefore, by \eqref{bound-I-v} we obtain
\begin{equation}\label{bound-J-v}
\J(\id+\e u^\e)-\frac43\pi\gamma a^2 \leq  C\e^2
\end{equation}
and
$$
\eta_\e\int_\Os |\nabla^2 u^\e|^p\, dx  \leq  C.
$$
By Sobolev embedding we deduce that
$$
\|\nabla u^\e\|_{C^0(\overline\Omega\setminus B_a)} \leq C\big( \|\nabla u^\e\|_{L^2}+ \|\nabla^2 u^\e\|_{L^p}\big) \leq  C + C\left(\frac{1}{\eta_\e}\right)^{\frac1p}.
$$
Thus, by \eqref{eta_e},
\begin{equation}\label{bound-ho-v}
\e\|\nabla u^\e\|_{C^0(\overline\Omega\setminus B_a)} \leq
C\left(\frac{\e^{p}}{\eta_\e}\right)^{\frac 1p}=:C\sigma_\e \to0.
\end{equation}

In all that follows the $O(\sigma_\e^3)$ terms should be understood as quantities whose norm in $C^0(\overline\Omega\setminus B_a)$ is of order $\sigma_\e^k$ with $k\ge 3$. In other words, because of  estimate \eqref{bound-ho-v}, we can neglect in the $\Gamma$-convergence process all terms that are more than quadratic in~$\e \nabla u^\e$.

Using \eqref{eq.alg1} we have that
\begin{equation}\label{eq.det-MG}
\det (I+\e\nabla u^\e)= 1+\e\dive u^\e+\e^2\tr(\cof\nabla u^\e)+\e^3\det\nabla u^\e,
\end{equation}
while
\begin{equation}\label{eq.cof-MG}
\cof (I+\e\nabla u^\e)= I+ \e\big((\dive u^\e) I-(\nabla u^\e)^T\big)+
\e^2\cof \nabla u^\e.
\end{equation}
From \eqref{eq.alg2} and \eqref{eq.alg3} we can further write that
\begin{multline}\label{eq.cof-MGbis}
\cof (I+\e\nabla u^\e)  =  I+ \e\big((\dive u^\e) I-(\nabla u^\e)^T\big) \\
+
\e^2 \big(\tr(\cof\nabla u^\e)I-\dive u^\e (\nabla u^\e)^T +(\nabla u^\e)^T(\nabla u^\e)^T\big).
\end{multline}

By \eqref{eq.cof-MG} we obtain
\begin{eqnarray}
|\cof (I+\e\nabla u^\e) \vece|^2 & = & 1+2\e\big(\dive u^\e -(\nabla u^\e)^T\vece\cdot\vece\big) \nonumber
\\
&&
{}+ \e^2\big( |(\dive u^\e)\vece-(\nabla u^\e)^T\vece|^2+
2\cof\nabla u_\e\vece\cdot\vece\big) \nonumber
\\ &&
{} +O(\sigma_\e^3). \label{fist}
\end{eqnarray}

The $\e$-term in \eqref{fist} reads as $2\e \dive_\tau u^\e$, while for the $\e^2$-term we use that
\begin{equation}\label{eq.form-tau}
(\dive u^\e)\vece-(\nabla u^\e)^T\vece= (\dive_\tau u^\e)\vece-(\nabla_\tau u^\e)^T\vece
\end{equation}
so that
$$
|(\dive u^\e)\vece-(\nabla u^\e)^T\vece|^2= |(\dive_\tau u^\e)\vece-(\nabla_\tau u^\e)^T\vece|^2
=(\dive_\tau u^\e)^2+ |(\nabla_\tau u^\e)^T\vece|^2.
$$
 Finally, from \eqref{eq.alg4}, we deduce that $\cof\nabla u_\e\vece=\cof\nabla_\tau u_\e\vece$.
Thus,
\begin{eqnarray*}
|\cof (I+\nabla u^\e) \vece|^2 & = & 1+2\e \dive_\tau u^\e
\\
& & + \e^2\big( |(\nabla_\tau u^\e)^T\vece|^2 +(\dive_\tau u^\e)^2
+2\cof\nabla_\tau u_\e\vece\cdot\vece\big) +O(\sigma_\e^3).
\end{eqnarray*}
Using the expansion $\sqrt{1+x}=1+\frac12 x-\frac18 x^2+O(x^3)$, we conclude that
\begin{equation}\label{eq.|cof|}|\cof \nabla y^\e \vece|=1+\e \dive_\tau u^\e
+ \frac{\e^2}2\big( |(\nabla_\tau u^\e)^T\vece|^2
+2\cof\nabla_\tau u_\e\vece\cdot\vece\big)
+O(\sigma_\e^3).
\end{equation}
From \eqref{eq.alg2} and \eqref{eq.alg3}, we have that
\begin{multline*}
2\cof\nabla_\tau u_\e\vece\cdot\vece=2\tr(\cof \nabla_\tau u^\e) -(\dive_\tau u^\e)(\nabla_\tau u^\e)^T\vece\cdot\vece\\+(\nabla_\tau u^\e)^T(\nabla_\tau u^\e)^T\vece\cdot\vece.
\end{multline*}
But, since $\nabla_\tau u^\e\vece=0$,
$$
-(\dive_\tau u^\e)(\nabla_\tau u^\e)^T\vece\cdot\vece+(\nabla_\tau u^\e)^T(\nabla_\tau u^\e)^T\vece\cdot\vece=0,
$$
so that, using \eqref{eq.alg3} once again,
 \begin{equation*}\label{eq.cofuerer}
 2\cof\nabla_\tau u_\e\vece\cdot\vece=2\tr(\cof \nabla_\tau u^\e)= (\dive_\tau u^\e)^2-(\nabla_\tau u^\e)^T\cdot\nabla_\tau u^\e.
 \end{equation*}

 Hence \eqref{eq.|cof|} finally reads as
 \begin{equation}\label{eq.|cof|bis}
 |\cof \nabla y^\e \vece|=1+\e \dive_\tau u^\e+\frac{\e^2}2\big( |(\nabla_\tau u^\e)^T\vece|^2
+(\dive_\tau u^\e)^2-(\nabla_\tau u^\e)^T\cdot\nabla_\tau u^\e\big)+O(\sigma_\e^3).
 \end{equation}

With the help of \eqref{eq.det-MG} and \eqref{eq.|cof|bis} we get from  \eqref{eq.tot-en2}
\begin{multline}\label{eq.lin-en}
\J(\id+\e u^\e)  =  \gamma \H(\p B_a)-p|B_a|+\e\int_{\p B_a}\big(\gamma\dive_\tau u^\e  -p u^\e\cdot\vece\big)\, d\H
\\ +\e^2
\left( \int_{\p B_a} \frac\gamma2\left(  |(\nabla_\tau u^\e)^T\vece|^2
+(\dive_\tau u^\e)^2-(\nabla_\tau u^\e)^T\cdot\nabla_\tau u^\e \right) \, d\H   \right.
\\
+\frac{\lfl}{2|B_a|} \left(\int_{\p B_a}  u^\e\cdot\vece\, d\H\right)^2
\left. {} -p\int_B\tr(\cof\nabla u^\e)\, dx\right)+O(\sigma_\e^3).
\end{multline}
In view of \eqref{eq.eqm}, the constant  term in \eqref{eq.lin-en} is $4/3\pi\gamma  a^2$ while the linear term disappears upon invoking  the second equality in \eqref{eq.der-area} for the sphere $\p B_a$.

Note also that, with the help of \eqref{eq.alg3} once again,
$$
\tr(\cof\nabla u^\e)=\frac12\big((\dive u^\e)^2- \tr [(\nabla u^\e)^2]\big)=\frac12\big(\dive(\dive u^\e u^\e)-\dive(\nabla u^\e u^\e)\big),
$$
so the last term in \eqref{eq.lin-en} can be written as the boundary term
$$
-\frac{p}2\int_{\partial B_a} u^\e \cdot(\dive u^\e\vece- (\nabla u^\e)^T \vece)\ d\H
$$
or still, in view of \eqref{eq.form-tau}, as
$$
 -\frac{p}2 \int_{\partial B_a}\big(\dive_\tau u^\e u^\e \cdot \vece-  u^\e\cdot (\nabla_\tau u^\e)^T\vece\big)\, d\H.
$$
Summing up, \eqref{eq.lin-en} also reads as
\begin{eqnarray}
\lefteqn{\J(\id+\e u^\e) -\J(\id)}
\nonumber \\
& = & \e^2
\left( \int_{\p B_a} \frac\gamma2\left(  |(\nabla_\tau u^\e)^T\vece|^2
+(\dive_\tau u^\e)^2-(\nabla_\tau u^\e)^T\cdot\nabla_\tau u^\e \right) \, d\H  \right.
\nonumber \\
& & {} +\frac{\lfl}{2|B_a|} \left(\int_{\p B_a}  u^\e\cdot\vece\, d\H\right)^2
\nonumber \\
& & \left. {} -\frac{p}2 \int_{\partial B_a}\big(\dive_\tau u^\e u^\e \cdot \vece-  u^\e\cdot (\nabla_\tau u^\e)^T\vece\big)\, d\H\right)+O(\sigma_\e^3).
\label{eq.lin-en-bis}
\end{eqnarray}
Appealing to \eqref{eq.tau-nabla-div},
 we obtain, after some algebraic manipulations,
\begin{multline*}
(\dive_\tau u^\e)^2-(\nabla_\tau u^\e)^T:\nabla_\tau u^\e=
 \dive_\tau(\dive_\tau u^\e u^\e - \nabla_\tau u^\e u^\e)
\\
{}+(\dive_\tau(\nabla_\tau u^\e)^T\cdot\vece)(u^\e\cdot\vece)
+(\nabla_\tau u^\e)^T\vece\cdot (\nabla_\tau\vece)^Tu^\e.
\end{multline*}
Therefore, using the second equality in \eqref{eq.der-area} and \eqref{eq.eqm},  \eqref{eq.lin-en-bis} reduces to
\begin{multline}\label{e2-red}
\J(\id+\e u^\e)-\J(\id)=\\
\e^2\left(\int_{\p B_a} \frac\gamma2\left(  |(\nabla_\tau u^\e)^T\vece|^2
+(\dive_\tau(\nabla_\tau u^\e)^T\cdot\vece)(u^\e\cdot\vece) +(\nabla_\tau u^\e)^T\vece\cdot (\nabla_\tau\vece)^Tu^\e
\right)  d\H\right. \\
{}+\frac{\lfl}{2|B_a|} \left(\int_{\p B_a}  u^\e\cdot\vece\, d\H\right)^2\left.\vphantom{\int_{\p B_a}}\right)+O(\sigma_\e^3).
\end{multline}
We now write $u^\e=v^\e+\varphi^\e\vece$, where $v^\e\cdot\vece=0$ and $\varphi^\e=u^\e\cdot\vece$.
By differentiating we have
$$
(\nabla_\tau v^\e)^T\vece= - (\nabla_\tau \vece)^Tv^\e= -\frac1a v^\e
$$
since $\nabla_\tau \vece=1/a(I-\vece\otimes\vece)$. Thus
\begin{equation}\label{eq.trans1}
(\nabla_\tau u^\e)^T\vece = \nabla_\tau \varphi^\e- (\nabla_\tau \vece)^Tv^\e =
\nabla_\tau \varphi^\e-\frac1a v^\e.
\end{equation}
Therefore,
\begin{eqnarray}
(\nabla_\tau u^\e)^T\vece\cdot (\nabla_\tau\vece)^Tu^\e & = & (\nabla_\tau u^\e)^T\vece\cdot (\nabla_\tau\vece)^Tv^\e
\nonumber
\\
& = & \label{eq.trans3}
(\nabla_\tau u^\e)^T\vece\cdot(\frac1a v^\e)=\frac1a \nabla_\tau \varphi^\e\cdot v^\e-\frac1{a^2} |v^\e|^2,
\end{eqnarray}
while, using that $\nabla_\tau u^\e\vece=0$ and that $\nabla_\tau \varphi^\e\cdot\vece=0$,
\begin{eqnarray}
\dive_\tau(\nabla_\tau u^\e)^T\cdot\vece & = & \dive_\tau(\nabla_\tau u^\e\vece)-(\nabla_\tau u^\e)^T:\nabla_\tau\vece
\nonumber
\\
& = & -(\nabla_\tau u^\e)^T\cdot\nabla_\tau\vece  \nonumber
\\
& = & -\frac 1 a\dive_\tau u^\e = -\frac1a \dive_\tau v^\e-\frac2{a^2}\varphi^\e. \label{eq.trans2}
\end{eqnarray}
Combining \eqref{eq.trans1}--\eqref{eq.trans2} together, expression \eqref{e2-red} can again be rewritten as
\begin{multline}\label{eq.Jfinal}
\J(\id+\e u^\e)-\J(\id)= \\
\e^2\left(\int_{\p B_a} \frac\gamma2\left(  \big|\nabla_\tau \varphi^\e-\frac1a v^\e\big|^2
-\frac1a \varphi^\e\dive_\tau v^\e-\frac2{a^2}|\varphi^\e|^2
+\frac1a \nabla_\tau \varphi^\e\cdot v^\e-\frac1{a^2} |v^\e|^2
\right)  d\H\right. \\
+\frac{\lfl}{2|B_a|} \left(\int_{\p B_a}  \varphi^\e\, d\H\right)^2 \left. \vphantom{\int_{\p B_a}}\right)+O(\sigma_\e^3).
\end{multline}
Integrating by parts the second term in the first integral above yields, in view of the second equality in \eqref{eq.der-area},
\begin{eqnarray*}
\int_{\p B_a} \varphi^\e\dive_\tau v^\e\, d\H  & = &
- \int_{\p B_a} \nabla_\tau\varphi^\e\cdot v^\e\, d\H +
\frac2a \int_{\p B_a} \varphi^\e v^\e\cdot\vece \, d\H
\\
& = & - \int_{\p B_a} \nabla_\tau\varphi^\e\cdot v^\e\, d\H.
\end{eqnarray*}
We finally conclude that expression \eqref{eq.Jfinal} is given by
\begin{multline}\label{eq.JFinal}
 \J(\id+\e u^\e)-\J(\id)= \J(\id+\e u^\e)- \frac4{3}\pi\gamma  a^2 \\
 =\e^2\left(\frac\gamma2\int_{\p B_a}|\nabla_\tau \varphi^\e|^2\, d\H
-\frac{\gamma}{a^2}\int_{\p B_a}  |\varphi^\e|^2 \, d\H +\frac{\lfl}{2|B_a|} \left(\int_{\p B_a}  \varphi^\e\,d\H\right)^2\right)+O(\sigma_\e^3),
\end{multline}
where we recall that $\varphi^\e=u^\e\cdot \vece$.

By \eqref{bound-J-v}, \eqref{eta_e} and the definition \eqref{bound-ho-v} of $\sigma_\e$, we deduce that
\begin{eqnarray*}
\frac\gamma2\int_{\p B_a}|\nabla_\tau (u^\e\cdot\vece)|^2\, d\H
& \leq &
\frac{\gamma}{a^2}\int_{\p B_a}  |u^\e\cdot\vece|^2 \, d\H
-\frac{\lfl}{2|B_a|} \left(\int_{\p B_a}  u^\e\cdot\vece\,d\H\right)^2
\\
&&
{}+ C+ \frac1{\e^2}O(\sigma_\e^3)  \\
& \leq &C+ C\|u^\e\|_{L^2(\p B_a)}^2+C \left(\frac{\e^{p/3}}{\eta_\e}\right)^{\frac 3p}\le C\|u^\e\|_{L^2(\p B_a)}^2+C .
\end{eqnarray*}
Since $(u_\e)$ is bounded in $H^1(\Os;\R^3)$, its trace on $\p B_a$ is bounded in $L^2(\p B_a;\R^3)$. Therefore, the inequality above implies that
\begin{equation}\label{eq.boundtang}(\nabla_\tau (u^\e\cdot\vece)) \mbox{ is bounded in }L^2(\p B_a;\R^3).\end{equation}

At the expense of  extracting a subsequence,  $u_\e\wto u$ weakly in $H^1(\Os;\R^3)$, $u_\e\to u$ strongly in $L^2(\p B_a;\R^3)$, and
$\nabla_\tau (u^\e\cdot\vece)\wto \nabla_\tau (u\cdot\vece)$ weakly in $L^2(\p B_a;\R^3)$, for some $u\in \A$, which completes the first part of the proof of Theorem~\ref{thm}.\medskip

\noindent {\sf Step~2: $\Gamma$-liminf.}
Assume that $u^\e\to u$ strongly in $L^2(\O;\R^3)$ and, also, without loss of generality that $\F_\e(u^\e)\le C$ and $\liminf \F_\e(u^\e)=\lim \F_\e(u^\e)$. The first step of this proof guarantees that
\eqref{eq.H1bound}, \eqref{bound-ho-v}, \eqref{eq.JFinal}, and \eqref{eq.boundtang} hold. Therefore, by
\eqref{bound-ho-v} and lower semicontinuity, we have that, up to a subsequence,
\begin{multline}\label{liminf1}
\liminf_{\e\to0}\frac1{\e^2}\Big(
\J(\id+\e u^\e) -\frac4{3}\pi\gamma  a^2
\Big) \\
\geq
\frac\gamma2\int_{\p B_a}|\nabla_\tau (u\cdot\vece)|^2\, d\H
-\frac{\gamma}{a^2}\int_{\p B_a}  (u\cdot\vece)^2 \, d\H
+\frac{\lfl}{2|B_a|} \left(\int_{\p B_a}  u\cdot\vece\,d\H\right)^2.
\end{multline}
Moreover,
\begin{equation}\label{liminf2}
\liminf_{\e\to0}
\left(\eta_\e\int_\Os |\nabla^2 u^\e|^p\, dx -\int_\Os f\cdot u^\e\, dx
\right)
\geq
-\int_\Os f\cdot u\, dx.
\end{equation}
Finally, as in \cite{DMNP},
$$
\liminf_{\e\to0}\frac1{\e^2}\int_\Os W(I+\e\nabla u^\e)\, dx
\geq
\frac12 \int_\Os Q(\bE u)\, dx.
$$
Note however that, in view of \eqref{bound-ho-v} and since the quadratic form $Q$ is positive definite  (see Remark~\ref{rm.positdef}), the above liminf is trivial in our setting.
Together with \eqref{liminf1} and \eqref{liminf2}, this proves that
$$
\lim_{\e\to0}\F_\e(u^\e)\ge \F(u).
$$
\medskip

\noindent {\sf Step~3:  $\Gamma$-limsup.}
We have to show that for every $u\in \A$ (see \eqref{def.A} for the definition of $\A$) there exists a sequence $(u^\e)\subset  C^2_{loc}(\O;\R^3)\cap W^{2,p}_\Dir(\O;\R^3)$
such that $u^\e$ converge to $u$ strongly in $H^1(\Os;\R^3)$, $u^\e\cdot\vece\to u\cdot\vece$ strongly in $H^1(\p B_a)$,
and
\begin{equation}\label{liminf}
\lim_{\e\to0}\F_\e(u^\e) = \F(u).
\end{equation}

If $u\in C^2_{loc}(\Omega\setminus B_a;\R^3)\cap W^{2,p}_\Dir(\O;\R^3)$, the constant sequence
$u^\e:=\tilde u$,
where $\tilde u$ is any $C^2_{loc}$ extension of $u$ to $\O$,
has all the desired properties. Indeed, since $p>3$, by Sobolev embedding $\nabla u$ is uniformly bounded in $\ol\Omega\setminus B_a$, so that by Taylor expansion
$$
\lim_{\e\to0}\frac1{\e^2}\int_\Os W(I+\e\nabla u)\, dx
=
\frac12 \int_\Os Q(\nabla u)\, dx
$$
and moreover \eqref{eq.JFinal} holds with $O(\e^3)$ in place of $O(\sigma_\e^3)$. Therefore,
\begin{multline*}
\lim_{\e\to0}\frac1{\e^2}\Big(
\J(\id+\e u^\e) -\frac4{3}\pi\gamma  a^2
\Big) \\
=
\frac\gamma2\int_{\p B_a}|\nabla_\tau (u\cdot\vece)|^2\, d\H
-\frac{\gamma}{a^2}\int_{\p B_a}  |u\cdot\vece|^2 \, d\H
+\frac{\lfl}{2|B_a|} \left(\int_{\p B_a}  u\cdot\vece\,d\H\right)^2.
\end{multline*}
Finally, since $\eta_\e\to 0$ by \eqref{eta_e},
$$
\eta_\e\int_\Os |\nabla^2 u|^p\, dx -\int_\Os f\cdot u\, dx
\to
 -\int_\Os f\cdot u\, dx.
$$

Let now $u\in \A$. By \cite[Lemmas~A.1 and~A.2]{ADMLP} and the assumptions on $\Dir$ there exists a sequence $(v_n)\subset C^\infty(\overline\O\setminus B_a;\R^3)$ such that $v_n=0$ on $\Gamma$ and $v_n$ converge to $u$ strongly in $H^1(\Os;\R^3)$. Let $\eta>0$ be such that $\ol B_{a+\eta}\subset\O$.
Consider a sequence $(\psi_n) \subset C^\infty(\p B_a)$ that approximates $u\cdot\vece$ strongly in $H^1(\p B_a)$
(see e.g. \cite{H}).

Let $U:=B_{a+\eta}\setminus\ol B_a$. Consider the solution $w_n$ of the following system:
$$\begin{cases}
-\Delta w_n = -\Delta (v_n\cdot\vece) \quad \mbox{ in } U,\\[2mm]
w_n= \psi_n  \mbox{ on }\p B_a,\quad w_n= v_n\cdot\vece \mbox{ on }\p B_{a+\eta}.
\end{cases}$$
By elliptic regularity we have that $w_n\in C^\infty(\ol U)$. Moreover,
$$
w_n\to w \quad \mbox{ strongly in } H^1(U),
$$
where $w$ is the solution to
$$\begin{cases}
-\Delta w= -\Delta (u\cdot\vece) \quad \mbox{ in } U,\\[2mm]
w= u\cdot\vece \quad\mbox{ on }\p U,
\end{cases}
$$
so that $w=u\cdot\vece$. Further, by construction $w_n$ converges strongly to $u\cdot\vece$ in $H^1(\p B_a)$.

Let $\varphi\in C^\infty_c(B_{a+\eta})$ be a cut-off function such that $\varphi=1$ on $B_{a+\eta/2}$. Define
$$
u_n:=\varphi (w_n- v_n\cdot\vece) \vece+  v_n.
$$
Then $u_n\in C^\infty(\ol\O\setminus B_a;\R^3)$ for every $n$, $u_n\to u$ strongly in $H^1(\Os;\R^3)$, and
$u_n\cdot\vece \to u\cdot\vece$ strongly in $H^1(\p B_a)$.
Since we clearly have that $\F(u_n)\to \F(u)$, the result is achieved through a diagonalization process.
\end{proof}

\begin{remark}[The linearized system]\label{rem.PDE}
The $\Gamma$-convergence and compactness result of Theorem~\ref{thm} immediately implies the existence of a minimizer $u$ for $\F$ on $\A$. Simple variations and use of the second equality in \eqref{eq.der-area} then demonstrate that $u$ satisfies the following set of equations:
\begin{equation}\label{eq.PDEs}
\begin{cases}
-\dive (\Aa\bE u)=f \quad \mbox{ in }\Os, \\[2mm]
\ds\gamma\Delta_\tau(u\cdot\vece)+2\frac\gamma {a^2} u\cdot\vece+\big((\Aa \bE u) \vece\big)_r-\frac{3\lfl }{4\pi a^3} \int_{\p B_a} u\cdot\vece\; d\H=0 \quad \mbox{ on }\p B_a,\\[3mm]
(\Aa \bE u)\vece\parallel \vece \quad \mbox{ on }\p B_a,\\[2mm]
u=0 \mbox{ on }\Dir,\quad (\Aa \bE u)\nu_{\p\O} =0 \mbox{ on }\p \O\setminus \ol\Dir,
\end{cases}
\end{equation}
where, again, $\bE u=1/2(\nabla u+\nabla u^T)$ and $\Delta_\tau$ is the Laplace-Beltrami operator, defined by $\Delta_\tau\varphi=\dive_\tau(\nabla_\tau\varphi)$, on $\p B_a$. Note that $(\Aa\bE u)\vece$ is an element of $H^{-1/2}(\p B_a;\R^3)$, so the notation $(\Aa \bE u)\vece\parallel \vece$ means that it only acts on the radial component of elements of $H^{1/2}(\p B_a;\R^3)$, while $\big((\Aa \bE u) \vece\big)_r$ is defined through the following equality:
 \begin{equation}\label{def.dual}
 \langle\big((\Aa \bE u) \vece\big)_r,v\cdot \vece\rangle:=\langle (\Aa \bE u) \vece,v\rangle_{H^{-1/2}(\p B_a)\times H^{1/2}(\p B_a)}
 \end{equation}
for every $v\in H^{1/2}(\p B_a;\R^3)$.

Note that \eqref{eq.PDEs} has a unique solution since the associated quadratic form, that~is,
\begin{multline*}
\frac12 \int_\Os Q(\bE u)\, dx
+\frac\gamma2\int_{\p B_a}|\nabla_\tau (u\cdot\vece)|^2\, d\H
-\frac{\gamma}{a^2}\int_{\p B_a}  |u\cdot\vece|^2 \, d\H
\\
+\frac{\lfl }{2|B_a|} \left(\int_{\p B_a}  u\cdot\vece\,d\H\right)^2
\end{multline*}
is coercive on $\A$ in view of the positive definiteness of
$Q$ (see Remark \ref{rm.positdef}) and of Remark \ref{rm.posit}. So, uniqueness and existence in $\A$ -- which we have already secured, thanks to the $\Gamma$-convergence process -- can be obtained directly through Lax-Milgram lemma.

This is, to our knowledge, the first time that a linearization process produces an interfacial PDE, this at the expense of introducing a vanishing regularization. The set of PDE's (\ref{eq.PDEs}) is precisely that derived formally in \cite{GLP} when specialized to solid/liquid interfaces with constant surface tension $\gamma$.\hfill\P
\end{remark}

\begin{remark}\label{rem.AEur}
 Testing \eqref{eq.PDEs} by $\vece$, we obtain with the help of \eqref{eq.der-area} that
 $$
 \langle \big((\Aa \bE u) \vece\big)_r,1\rangle= \left(3\frac{\lfl}{a}-2\frac\gamma{a^2}\right)\int_{\p B_a} u\cdot\vece\; d\H.
 $$
\vskip-.9cm \hfill\P
\end{remark}

\begin{remark}\label{rem.manyinclusions}
The above result equally applies to domains containing any finite number of liquid inclusions.\hfill\P
\end{remark}

\section{The linearized problem in the presence of many inclusions}\label{sec.hom}

\subsection{Homogenization.}
In this subsection, we propose to investigate the limit (ma\-croscopic) behavior of a linearized solid filled with many periodically distributed liquid inclusions pressurized at the same pressure. Note that the periodicity assumption is not essential; we adopt it below for brevity sake.

The setting is as follows. The domain $\O$ of the previous sections is under the same loading $f$ (defined as an element of  $L^2(\O;\R^3)$ this time) and boundary conditions as before (see the beginning of Section \ref{sec.lin}).

We cover $\O$ with identical disjoint cubes $Y^i_\e:=\e i+\e Y$ for $i\in\Z^3$, $Y:=[-1/2,1/2)^3$, each containing an identical centered spherical inclusion $\Bie:=\e i+\e  B_a$ with $a<1/2$, filled with a liquid pressurized at the pressure $\e p$.
Let $\Ie$ denote the set of centers $i\in\Z^3$ such that $Y^i_\e\subset\Omega$ and ${\rm dist}(Y^i_\e,\partial\Omega)\geq\e$. Note that $\#(I_\e)\simeq1/\e^3 $.
We define the following  domains
$$
\Oe:=\O\setminus\left(\cup_{i\in\Ie} \ol B_{\e a}^i\right), \quad\oe:=\cup_{i\in\Ie}(Y^i_\e\setminus \ol B^i_{\e a}), \quad\K:=\cup_{i\in\Ie}Y^i_\e.$$
As an immediate corollary of the results  in Remark \ref{rem.PDE}, the solution $\ue$ to the system
\begin{equation}\label{eq.eps-sys}
\begin{cases}
-\dive (\Aa\bE \ue)=f \quad \mbox{ in }\Oe,\\[2mm]
\ds\e\gamma\Delta_\tau(\ue\cdot\vece)+2\frac{\gamma} {\e a^2} \ue\cdot\vece+((\Aa \bE \ue)\vece)_r
\\\qquad\qquad\qquad\qquad\displaystyle
-\frac{3\lfl }{4\pi \e^3a^3} \int_{\p \Bie} u\cdot\vece\; d\H=0 \quad\mbox{ on }\p \Bie, \ i\in\Ie, \\[3mm]
(\Aa \bE \ue)\vece\parallel \vece \quad \mbox{ on }\p \Bie, \ i\in \Ie, \\[2mm]
u=0 \mbox{ on }\Dir,\quad (\Aa \bE \ue)\nu_{\p\O} =0 \mbox{ on }\p \O\setminus \ol\Dir
\end{cases}
\end{equation}
exists and is unique in the class
\begin{equation}\label{def.Ae}
\A^\e:=\left\{u\in H^1(\Oe;\R^3): \ u\cdot\vece\in \cup_{i\in\Ie} H^1(\p \Bie) \text{ and }u=0 \text{ on } \Dir\right\}.
\end{equation}

Note that the various powers of $\e$ in \eqref{eq.eps-sys} correspond to the rescaling of both $p$ and $\gamma$ by $\e$, the natural scaling if one wishes to conserve both \eqref{eq.eqm} and \eqref{eq.cond-min}.

The solution $\ue$ is then the (unique) minimizer in $\A^\e$ of the functional
\begin{equation}\label{eq.def-fe}
\F^\e(v):=\frac12 \int_{\Oe} Q(\bE v)\, dx
+\sum_{i\in\Ie } \mathcal V^\e_i(v) -\int_{\Oe}f\cdot v\, dx,
\end{equation}
where
\begin{multline}\label{Vei-var-0}
\mathcal V^\e_i(v):=\frac{\gamma\e}2\int_{\p \Bie}|\nabla_\tau (v\cdot\vece)|^2\, d\H
{} -\frac{\gamma}{\e a^2}\int_{\p \Bie}  (v\cdot\vece)^2 \, d\H
\\
{}+\frac{\lfl }{2\e^3|B_a|} \left(\int_{\p \Bie}  v\cdot\vece\,d\H\right)^2
\end{multline}
for $i\in\Ie$. Note that by \eqref{identity} with $K={\lfl a^2/ (2\gamma |B_a|)}$ we can rewrite
\begin{multline}\label{Vei-var}
\mathcal V^\e_i(v)=\frac{\gamma\e}2\int_{\p\Bie}|\nabla_\tau (P^2_{i,\e a}(v\cdot \vece  ))|^2\,d{\mathcal  H}^2
- \frac{\gamma}{\ep a^2}\int_{\p\Bie}|P^2_{i,\e a}(v\cdot \vece)  |^2\,d{\mathcal  H}^2
\\
{} + \frac1{4\pi\e^3 a^3}\Big(\frac{3\lfl}2-\frac\gamma a\Big)\left(\int_{\p\Bie}v\cdot \vece  \,d{\mathcal  H}^2\right)^2,
\end{multline}
where  $P^2_{i,\e a}$ is the orthogonal projection in $L^2(\p \Bie)$ onto the orthogonal space to affine functions, see Appendix~A.

From the minimality of $\ue$ we have
\begin{eqnarray}
0=\F^\e(0) &\ge & \F^\e(\ue)  \nonumber
\\
& \ge &
\frac12 \int_{\Oe} Q(\bE \ue)\, dx +
\sum_{i\in\Ie}\left\{\frac{\gamma\e}3\int_{\p \Bie}|\nabla_\tau (P^2_{i,\e a}(\ue\cdot\vece))|^2\, d\H +\right. \nonumber
\\
& &\left. {\frac1 {4\pi\e^3 a^3}}\Big(\frac{3\lfl}2-\frac\gamma a\Big) \left(\int_{\p \Bie}  \ue\cdot\vece\,d\H\right)^2\right\}-\int_{\Oe}f\cdot \ue\, dx,
\label{eq.est.feue}
\end{eqnarray}
where we used \eqref{Vei-var} and the coercivity estimate \eqref{coer} in Appendix~A.
By \eqref{eq.cond-min} and because of the positive definite character of $Q$ we deduce that
\begin{equation}\label{eq.bd-ue}
\|\bE \ue\|^2_{L^2(\Oe)}\le C \|\ue\|_{L^2(\Oe)}.
\end{equation}
Appealing to \cite[Theorem~4.2]{OSY}, there exists a linear extension operator $$R^\e: H^1(\Oe;\R^3)\to H^1(\O;\R^3)$$ such that, for some constant $C$ independent of $\e$,
$$
\begin{array}{c}
\|R^\e u\|_{H^1(\O)} \le C\|u\|_{H^1(\Oe)}, \\[2mm]
\|\bE R^\e u\|_{L^2(\O)} \le C\|\bE u\|_{L^2(\Oe)}
\end{array}
$$
for every $u\in H^1(\Oe;\R^3)$.
Because of this result we actually obtain, in lieu of \eqref{eq.bd-ue}, that
$$
\|\bE R^\e\ue\|^2_{L^2(\O)}\le C \|R^\e\ue\|_{L^2(\O)},
$$
so that, using Korn and Poincar\'e-Korn inequalities on $\O$, we conclude that
\begin{equation}\label{eq.bd-ue-b}
\|R^\e\ue\|_{H^1(\O)}\le C.
\end{equation}
Further, by \eqref{eq.est.feue} and \eqref{eq.bd-ue-b} we deduce that
\begin{equation}\label{eq.bd-ue-b2}
\sum_{i\in\Ie}\left(\int_{\p \Bie}  \ue\cdot\vece\,d\H\right)^2\le C\e^{3}
\end{equation}
and
\begin{equation}\label{eq.bd-nabue}
\e\sum_{i\in\Ie} \int_{\p \Bie}|\nabla_\tau (P^2_{i,\e a}(\ue\cdot\vece))|^2\, d\H\le C.
\end{equation}

\medskip

Let
$$
H^1_\Dir(\O;\R^3):=\{u\in H^1(\O;\R^3): \ u=0\mbox{ on }\Dir\}.
$$
We propose to establish the following homogenization result.

\begin{theorem}\label{thm.hom}
The unique solution $\ue$ to \eqref{eq.eps-sys} can be extended to a function $R^\e\ue\in H^1(\O;\R^3)$ such that
$$
R^\e\ue \rightharpoonup u \quad \mbox{ weakly in } H^1(\O;\R^3),
$$
where $u$ is the unique solution in $H^1_\Dir(\O;\R^3)$ of
\begin{equation}\label{eq.hom-pb}
\begin{cases}
-\dive (\Aah\bE u)=(1-|B_a|)f \quad \mbox{ in }\O,\\[2mm]
u=0 \mbox{ on }\Dir, \quad (\Aah \bE u)\nu_{\p\O} =0 \mbox{ on }\p \O\setminus \ol\Dir,
\end{cases}
\end{equation}
with
\begin{equation}\label{eq.def-Ah}
\Aah F\cdot F :=
2\F_{per}(F, \chiF)
\end{equation}
for every $F\in \msym$ and $\chiF$ denotes the unique  minimizer of $\F_{per}(F,\cdot)$ defined in \eqref{eq.cell-pb} below.

Furthermore, the following corrector results {hold:} 
\begin{equation}\label{eq.resuco}
\lim_\e\int_{\omega\cap\omega^\e}\left|\bE_x\ue- \bE_x u- \frac1{\e^3}\sum_{i,j=1}^{3}\left(\int_{Y^{\kappa(x/\e)}_{\e}}(\bE_x u)_{ij}(z)\,dz\right) \bE_y\chiij(x/\e)\right|^2 dx=0
\end{equation}
for any $\omega\subset\subset\O$, and
\begin{multline}\label{eq.resuco2}
\lim_\e \e\sum_{i\in I_\e}\int_{\partial \Bie}\Bigg|\nabla_\tau (P^2_{i,\e a}( u^\ep\cdot \vece)) \\[2mm]
\left.-\frac1{\e^3}\nabla_\tau \Bigg(P^2_{a}\Bigg(\int_{Y^i_\e}\Big(\nabla_xu(z) \, y
+\sum_{j,k=1}^{3}(\bE_x u)_{jk}(z)\chijk(y)\Big)dz\cdot \vece\Bigg)\Bigg)\!\Big|_{y=x/\e}\right|^2d\H=0,
\end{multline}
where $P^2_a$ is the orthogonal projection in $L^2(\p B_a)$ onto  the orthogonal space to affine functions on $\p B_a$ while $P^2_{i,\e a}$ is the orthogonal projection in $L^2(\p \Bie)$ onto  the orthogonal space to affine functions on $\p \Bie$ (see \eqref{proj-e} in  Appendix~A) and
\begin{equation}\label{eq.def-chiij}
\chiij:=\chiFij
\end{equation}
with $(F_{ij})_{kh}=1/2(\delta_{ik}\delta_{jh}+\delta_{ih}\delta_{jk})$.
\end{theorem}

In Theorem~\ref{thm.hom} above,  the cell problem is given by
{
$$
\min\left\{ \F_{per}(F,\psi): \ \psi \in \X\right\},
$$
where
\begin{equation}\label{eq.defXcell}
\X:=\{\psi\in H^1_\sharp(Y\setminus\ol B_a;\R^3): \psi\cdot\vece\in H^1(\p B_a)\}
\end{equation}
and}
\begin{multline}\label{eq.cell-pb}
\F_{per}(F,\psi):= \frac12 \int_{Y\setminus\ol B_a}  Q(\bE \psi+F)\, dy
+\frac\gamma2\int_{\p B_a}|\nabla_\tau P^2_a  ((\psi+Fy)\cdot\vece)|^2\, d\H
\\
-\frac{\gamma}{a^2}\int_{\p B_a}  |P^2_a  ((\psi+Fy)\cdot\vece)|^2 \, d\H
\\
+\frac{1}{2|B_a|}\left(\lfl-\frac{2\gamma}{3a}\right) \left(\int_{\p B_a}  (\psi+Fy)\cdot\vece\,d\H\right)^2
\end{multline}
and $P^2_a  $ is the orthogonal projection in $L^2(\p B_a)$ onto  the orthogonal space to affine functions on $\p B_a$ (see \eqref{proj} in  Appendix~A).

\begin{remark}\label{rk.eqform}
As could be easily seen from reproducing the computations leading to \eqref{identity} in  Appendix~A, an equivalent expression for $\F_{per}$ defined in \eqref{eq.cell-pb} above is
\begin{multline*}
\F_{per}(F,\psi)= \frac12 \int_{Y\setminus\ol B_a}  Q(\bE \psi+F)\, dy
+\frac\gamma2\int_{\p B_a}|\nabla_\tau ((\psi+Fy)\cdot\vece)|^2\, d\H
\\[2mm]-\frac{\gamma}{a^2}\int_{\p B_a}  ((\psi+Fy)\cdot\vece)^2 \, d\H
+\frac{\lfl}{2|B_a|} \left(\int_{\p B_a}  (\psi+Fy)\cdot\vece\,d\H\right)^2.
\end{multline*}
In that form it is clear that an argument analogous to that used in Remark~\ref{rm.posit} would demonstrate the existence and uniqueness of $\chiF$.
We also note that, because of Remark \ref{rm.posit} and of \eqref{eq.cond-min},
$$\Aah \mbox{ defined in \eqref{eq.def-Ah} is definite positive.}$$
\vskip-.7cm\hfill\P\end{remark}

\begin{remark}\label{rk.corrector}
Since 
$$\nabla u- \frac1{\e^3}\int_{Y^{\kappa(x/\e)}_{\e }}\nabla u(z)dz\stackrel{\e}{\longrightarrow} 0\mbox{ strongly in }L^2(\Omega;\mthree),$$
a simpler expression for the corrector result \eqref{eq.resuco} can be obtained, namely
\begin{multline*}
\lim_\e\Bigg\{\int_{\omega\cap\omega^\e}\left|\bE_x\ue- \bE_x u- \sum_{i,j=1}^{3}(\bE_x u)_{ij} \bE_y\chiij(x/\e)\right|^2 dx \\
+\e\sum_{i\in I_\e}\int_{\partial \Bie}\Bigg|\nabla_\tau (P^2_{i,\e a}( u^\ep\cdot e_r))  \\[2mm] 
- \nabla_\tau\Bigg(P^2_{a}\left(\Big(\nabla_xu \, y
+\sum_{j,k=1}^{3}(\bE_x u)_{jk}\chijk(y)\Big)\cdot \vece\right)\Bigg)\!\Big|_{y=x/\e}\Bigg|^2d\H
\Bigg\}=0,
\end{multline*}
provided that, either {$\bE_y \lambda_{ij}\in L^\infty(Y;\msym)$ and $\nabla_\tau\lambda_{ij}\in L^\infty(\p B_a;\R^3)$} for all $i,j \in \{1,2,3\}$, or that $u$ turns out to be sufficiently smooth.

While the regularity of $u$ will hinge, in particular, on the regularity of $f$, the $L^\infty$ regularity of $\bE_y\lambda_{ij}$ and of $\nabla_\tau\lambda_{ij}$ might be true, but we confess a lack of determination in the matter.  
 \hfill\P\end{remark}

We now proceed with the proof of Theorem~\ref{thm.hom}.

\begin{proof}[Proof of Theorem~\ref{thm.hom}]
We  define the unfolding of $\ue$ adapting the ideas in \cite{ADH,CDG} (see also \cite{Cas1}).
Define $\kappa:\R^3\to \Z^3$ so that
$$
x\in Y^{\kappa(x)}_1,
$$
that is, $\kappa(x)$ provides the center $i\in \Z^3$ of the cube $Y^i_1:=i+Y$ containing $x$.
In particular,
$$
x\in Y^i_\e \quad \text{ if and only if } \quad i=\kappa\Big({\frac x \e}\Big)
$$
for every $x\in \R^3$.
We define the unfolding $\hue:\K\times (Y\setminus \ol B_a)\to \R^3$~as
\begin{equation*}\hue(x,y)=\ue\Big(\e\kappa\Big({\frac x\e}\Big)+\e y\Big).
\end{equation*}
Observe that, for $x\in Y^i_\e$, $\hue$ does not depend on $x$, while as a function of $y$, it just comes from $\ue$ by the change of variables
$\ds y=(x-\e i)/\e$,
which transforms $Y^i_\e\setminus \bBie$ into  $Y\setminus \ol B_a$.

Using the definition of $\hue$ and recalling that $\oe=\cup_{i\in\Ie}(Y^i_\e\setminus \ol B^i_{\e a})$, we have
\beq\label{eq.unfol-grad}
\ba{l}\displaystyle\int_{\oe} |\nabla \ue(x)|^2\,dx=\sum_{i\in\Ie }\int_{Y^i_\e\setminus \bBie}|\nabla \ue(x)|^2\,dx=\e^3\sum_{i\in\Ie }\int_{Y\setminus \ol B_a}|\nabla \ue(\e i +\e y)|^2\,dy
\\[3mm] \displaystyle
= \frac{1}{\e^2}\sum_{i\in\Ie }\int_{Y^i_\e}\int_{Y\setminus \ol B_a}|\nabla_y \hue(x,y)|^2\,dy\,dx=
\frac1{\e^2}\int_{\K}\int_{Y\setminus \ol B_a}|\nabla_y \hue(x,y)|^2\,dy\,dx.
\ea
\eeq
Analogously, we obtain that
\begin{eqnarray}\label{eq.unfol-ex}
\int_{\oe} Q(\bE \ue)\,dx & = & \frac1{\e^2}\int_{\K}\int_{Y\setminus \ol B_a}Q(\bE_y \hue(x,y))\,dy\,dx
\\
\sum_{i\in\Ie }\!\int_{\p\Bie}\!|P^2_{i,\e a}(\ue\cdot \vece)  |^2d\H \!\!& = &\!\! \frac1{\e}\int_{\K}\int_{\partial B_a}\!|P^2_a  (\hue(x,y)\cdot\vece(y))|^2\,d\H\!(y)dx
\\
\sum_{i\in\Ie }\left(\int_{\p\Bie}\ue\cdot\vece\, d\H\right)^2\!\! & = & \e\int_{\K}\!\!\left(\int_{\partial B_a} \hue(x,y)\cdot\vece(y)\, d\H(y)\right)^2\!\!dx,
\end{eqnarray}
where $\vece(y):=y/{|y|}$ and $P^2_a  (\hue(x,\cdot)\cdot\vece)$ denotes the orthogonal projection in $L^2(\p B_a)$ of the function $y\mapsto \hue(x,y)\cdot\vece(y)$
onto  the orthogonal space to affine functions on $\p B_a$. Moreover,
\beq\label{eq.unfol-ex2}
\sum_{i\in\Ie }\int_{\p\Bie}|\nabla_\tau (P^2_{i,\e a}(\ue\cdot\vece))|^2\,d\H =\frac1{\e^3}\int_{\K}\int_{\partial B_a}|\nabla_{\tau,y} (P^2_a  (\hue(x,\cdot)\cdot\vece))|^2\,d\H(y)\, dx.
\eeq

In view of \eqref{eq.bd-ue-b}--\eqref{eq.bd-nabue}, we conclude in particular that
\begin{multline*}
\frac1{\e^2} \int_{\K}\int_{Y\setminus \ol B_a}|\nabla_y \hue(x,y)|^2\,dy\,dx
+\frac1{\e^2} \int_{\K}\int_{\partial B_a}|\nabla_{\tau,y} (P^2_a  (\hue(x,\cdot)\cdot\vece))|^2\,d\H(y)\, dx
\\
+\frac1{\e^2} \int_{\K}\left(\int_{\partial B_a} \hue(x,y)\cdot\vece(y)\, d\H(y)\right)^2\!\!dx \le C.
\end{multline*}
Let now
\begin{eqnarray}
\hwe(x,y) & := & \frac1\e\hue(x,y)- \frac1\e \fint_{Y\setminus \ol B_a}\hue(x,z)\, dz
\nonumber \\
& = & \frac1\e \hue(x,y) - \frac1\e \fint_{Y^{\kappa(\frac x\e)}_{\e }\setminus \ol B^{\kappa(\frac x\e)}_{\e a}}\ue(z)\, dz. \label{eq.unfol}
\end{eqnarray}
By the previous bounds and Poincar\'e-Wirtinger's inequality applied to $Y\setminus \ol B_{a}$,
a (not relabeled) subsequence of $\hwe$
satisfies
\begin{equation}\label{eq.lim.uxy}
\hwe \rightharpoonup \hat w \quad \mbox{ weakly in } L^2(\omega; H^1(Y\setminus\ol B_a;\R^3))
\end{equation}
and
\begin{equation}\label{eq.lim.uxy2}
P^2_a  (\hwe\cdot\vece) \rightharpoonup P^2_a  (\hat w\cdot\vece) \quad \mbox{ weakly in } L^2(\omega; H^1(\p B_a))
\end{equation}
for any open set $\omega\subset\subset \O$ and for some $\hat w\in L^2_{\rm loc}(\O;H^1(Y\setminus\ol B_a;\R^3))$ such that $\hat w\cdot\vece\in L^2_{\rm loc}(\O; H^1(\p B_a))$.

Further,  in view of \eqref{eq.bd-ue-b}, we can assume that
\begin{equation}\label{eq.conv=Rue}
R^\e\ue \rightharpoonup u \quad \mbox{ weakly in } H^1(\O; \R^3).
\end{equation}

Take  $\vec{e}_1$ as the first vector of the canonical basis in $\RR^3$ and note that
the definition of $\hat u_\ep$ implies that
$$
\hat u_\ep\Big(x+\ep \vec e_1,-{\frac1 2},y_2,y_3\Big)=\hat u_\ep\Big(x,{\frac 1 2  },y_2,y_3\Big)$$
for a.e.\ $x\in \K$ and a.e.\ $(y_2,y_3)\in (-1/2,1/2)^2$.
Thus
the definition \eqref{eq.unfol} of $\hwe$  implies in turn that
\begin{multline}\label{eq.interm10}
\hat w_\ep\Big(x+\ep \vec e_1,-{\frac 1 2  },y_2,y_3\Big)-\hat w_\ep\Big(x,{\frac 1 2  },y_2,y_3\Big)
\\
=-\fint_{Y^{\kappa(\frac x\e)}_{\e }\setminus \ol B^{\kappa(\frac x\e)}_{\e a}}\frac{R^\ep u^\ep(z+\ep\vec{e}_1)-R^\ep u^\ep (z)}\e\, dz
\end{multline}
for a.e.\ $x\in \K$ and a.e.\ $(y_2,y_3)\in (-1/2,1/2)^2$.
Thus,  passing to the limit, we get
\beq\label{eq.interm1}
\hat w\Big(x,-{\frac 1 2  },y_2,y_3\Big)-\hat w\Big(x,{\frac 1 2},y_2,y_3\Big)=-\frac{\p u}{\p x_1}(x).
\eeq
Indeed, taking $\ph\in C^\infty_c(\Omega;\R^3)$ and integrating \eqref{eq.interm10}
over the support $K$ of $\ph$ we get, for $\e$ small enough,
\begin{eqnarray*}
\lefteqn{\int_K \Big(\hat w_\e\Big(x+\e \vec{e}_1,-{\frac 1 2  },y_2,y_3\Big)-\hat w_\ep\Big(x,{\frac 1 2  },y_2,y_3\Big)\Big)\ph(x)\, dx}
\\
& = & \int_{\K} \Big(\hat w_\e\Big(x+\e \vec{e}_1,-{\frac 1 2  },y_2,y_3\Big)-\hat w_\ep\Big(x,{\frac 1 2  },y_2,y_3\Big)\Big)\ph(x)\, dx
\\
& = & - \sum_{i\in\Ie} \int_{Y^i_{\e}}\fint_{Y^i_{\e }\setminus \ol B^i_{\e a}}\frac{R^\ep u^\ep(z+\ep\vec{e}_1)-R^\ep u^\ep (z)}\e \ph(x) \,dz\,dx
\\
& = & - \sum_{i\in\Ie} \int_{Y^i_{\e}}\fint_{Y^i_{\e }\setminus \ol B^i_{\e a}}\frac{R^\ep u^\ep(z+\ep\vec{e}_1)-R^\ep u^\ep (z)}\e \ph(z) \,dz\,dx +O(\e)
\\
& = & -\frac{1}{1-|B_a|}\int_{\oe}\frac{R^\ep u^\ep(z+\ep\vec{e}_1)-R^\ep u^\ep (z)}\e \ph(z) \,dz +O(\e)
\\
& = & -\frac{1}{1-|B_a|}\int_{\oe}\frac{\ph(z-\ep\vec{e}_1)- \ph (z)}\e R^\e u^\e(z) \,dz +O(\e).
\end{eqnarray*}
Since by periodicity $\chi_{\oe}\rightharpoonup 1-|B_a|$ weakly$^*$ in $L^\infty(\omega)$ for every $\omega\subset\subset\Omega$ and $R^\e u^\e\to u$ strongly in $L^2(\Omega;\R^3)$ by \eqref{eq.conv=Rue} and Rellich Theorem,
this yields \eqref{eq.interm1}.

Reasoning analogously with respect to the other vectors of the canonical basis, we conclude that
\begin{equation}\label{eq.def-u_1}
\hat u_1(x,y):=\hat w(x,y)-\nabla u(x)y \in L^2_{\rm loc}(\O;H^1_\sharp (Y\setminus \ol B_a;\R^3)).
\end{equation}

We now consider $v\in C^\infty(\ol\Om;\R^3)$ with $v=0$ on $\Dir$ and $v_1\in C^1_c(\Om;C^1_\sharp(Y\setminus \ol B_a;\R^3))$. Define $v^\ep$ by
\beq\label{eq.ve}
v^\ep(x)=v(x)+\ep v_1\Big(x,{\frac x\e}\Big).\eeq
By minimality we have
\begin{equation}\label{eq.uemi1}
{\mathcal  F}^\ep(u^\ep) \leq {\mathcal  F}^\ep(v^\ep).
\end{equation}

By definition of ${\mathcal  F}^\ep$ and \eqref{Vei-var}, we have
\begin{eqnarray}
{\mathcal  F}^\ep(v^\ep) & = & \frac12 \int_{\Om^\ep} Q(\bE v^\ep)\,dx
+\sum_{i\in\Ie} \frac{\gamma \ep}{2}\int_{\p\Bie}|\nabla_\tau (P^2_{i,\e a}(v^\ep\cdot \vece  ))|^2\,d{\mathcal  H}^2
\nonumber \\
&& {} - \sum_{i\in\Ie} \frac{\gamma}{\ep a^2}\int_{\p\Bie}|P^2_{i,\e a}(v^\ep\cdot \vece)  |^2\,d{\mathcal  H}^2
\nonumber \\
&&  {}+ \sum_{i\in\Ie} \frac1{4\pi\e^3 a^3}
 \Big(\frac{3\lfl}2-\frac\gamma a\Big)\left(\int_{\p\Bie}v^\ep\cdot \vece  \,d{\mathcal  H}^2\right)^2
 \nonumber \\
&&  {}
 -\int_{\Om^\ep}f\cdot v^\ep\, dx. \label{eq.terFev}
 \end{eqnarray}
 Let us pass to the limit in the different terms of the right-hand side of \eqref{eq.terFev}.
 The first term yields with obvious notation
 \begin{eqnarray}
\frac12 \int_{\Om^\ep} Q(\bE v^\ep)\,dx & = & \frac12 \int_{\Om^\ep} Q(\bE_xv(x)+\bE_yv_1(x,x/\e))\,dx+O(\e)
\nonumber \\
& = & \frac12 \sum_{i\in\Ie} \int_{Y^i_{\e }\setminus \ol B^i_{\e a}}Q(\bE_xv(\e i)+\bE_yv_1(\e i,x/\e))\,dx+O(\e) \nonumber
\\
& = & \frac12 \sum_{i\in\Ie} \e^3\int_{Y\setminus \ol B_a}Q(\bE_xv(\e i)+\bE_yv_1(\e i,y))\,dy+O(\e) \nonumber
\\
& = & \frac12 \int_{\K} \int_{Y\setminus B_a}Q(\bE_xv(x)+\bE_yv_1(x,y))\,dy\,dx+O(\e). \label{eq.terFev1}
 \end{eqnarray}
For the terms in \eqref{eq.terFev} on the boundary of the balls $\Bie$ we write
$$
v^\e(x)= v(\e i)+\nabla_x v(\e i)(x-\e i)+\e v_1(\e i, x/\e)+\omega^i_\e(x)
$$
for $x\in\partial\Bie$, where $\|\omega^i_\e\|_{C^0(\p\Bie)}=O(\e^2)$ and $\|\omega^i_\e\|_{C^1(\p\Bie)}=O(\e)$.
Using that $P^2_{i,\e a}(v(\e i)\cdot\vece)=0$, we obtain for the second term in \eqref{eq.terFev}
 \begin{eqnarray}
 \lefteqn{\frac{\gamma \ep}2\sum_{i\in\Ie}\int_{\partial \Bie}|\nabla_\tau (P^2_{i,\e a}(v^\ep\cdot \vece))|^2\,d\H}
 \nonumber \\
& =&
\frac{\gamma \ep^2}2\sum_{i\in\Ie}\int_{\p\Bie}\big|\nabla_{\tau,x}(P^2_{i,\e a}(a\nabla_x v(\ep i)\vece\cdot \vece  +v_1(\e i,x/\e)\cdot \vece)) \big|^2\,d\H+O(\e)
\nonumber \\
& = & \frac\gamma{2}\int_{\K}\int_{\partial B_a} \big|\nabla_{\tau,y}(P^2_a  (\nabla_x v\, y \cdot \vece  + v_1\cdot\vece)) \big|^2\,d\H(y)\, dx+O(\e). \label{eq.terFev2}
 \end{eqnarray}
Arguing in a similar way, the third term can be written as
 \begin{eqnarray}
 \lefteqn{\frac{\gamma}{\ep a^2}\sum_{i\in\Ie}\int_{\p\Bie}|P^2_{i,\e a}(v^\ep\cdot \vece)  |^2\,d\H}
 \nonumber \\
 & = & \frac{\gamma}{ a^2}\sum_{i\in\Ie}\int_{\p\Bie}\Big|P^2_{i,\e a}(a\nabla_x v(\ep i)\vece\cdot \vece+v_1(\e i, x/\e)\cdot\vece)  \Big|^2\,d\H +O(\e)
\nonumber \\
 & = &
\frac{\gamma}{a^2}\int_{\K}\int_{\partial B_a}\big|P^2_a  (\nabla_x v\, y\cdot   \vece +v_1\cdot\vece)  \big|^2\,d\H(y)\, dx+O(\e). \label{eq.terFev3}
\end{eqnarray}
For the fourth term we get
\begin{eqnarray}
\lefteqn{\frac1{4\pi\e^3 a^3}\sum_{i\in\Ie}\left(\int_{\partial \Bie}v^\ep\cdot \vece  \,d\H\right)^2}
\nonumber
\\
& = &\frac1{4\pi\e a^3}\sum_{i\in\Ie}\left(\int_{\p\Bie}(a\nabla_x v(\ep i)\vece\cdot\vece+ v_1(\e i,x/\ep)\cdot \vece)   \, d\H\right)^2+O(\e)
\nonumber \\
& = &\frac1{4\pi a^3}\int_{\K}\left(\int_{\partial B_a}\big(\nabla_x v(x) y\cdot   \vece +v_1(x,y)\cdot \vece  \big)\,d\H(y)\right)^2\!\! dx+O(\ep).
\label{eq.terFev4}
\end{eqnarray}
Finally, since by periodicity $\chi_{\oe\cap\omega}\rightharpoonup (1-|B_a|)\chi_\omega$ weakly in $L^2(\O)$ for every $\omega\subset\subset\Omega$, it is easily concluded, upon letting $\omega\nearrow\O$, that
\beq\label{eq.terFev5}
\int_{\Om^\ep}f\cdot v^\ep\, dx=\int_{\Om^\ep}f\cdot v \,dx+O(\e)\longrightarrow\big(1-|B_a|\big)\into f\cdot v\,dx.
\eeq

Collecting \eqref{eq.terFev1}--\eqref{eq.terFev5} and letting $\e$ tend to $0$, we finally obtain that, for $v^\e$ as in \eqref{eq.ve},
\begin{eqnarray}
\lim_\e {\mathcal  F}^\ep(v^\ep) &= &
\frac12 \into \int_{Y\setminus B_a} Q(\bE_xv(x)+\bE_yv_1(x,y))\,dy\,dx
\nonumber \\
&&
{}+\frac\gamma{2}\into\int_{\partial B_a} \big|\nabla_{\tau,y}(P^2_a  ( \nabla_x v(x)\, y\cdot   \vece +v_1(x,y)\cdot\vece))  \big|^2\, d\H(y)\, dx
\nonumber \\
&&
{}-\frac\gamma{a^2}\int_{\O}\int_{\partial B_a}\big|P^2_a  (\nabla_x v(x)\, y\cdot   \vece +v_1(x,y)\cdot\vece)\big|^2\,d\H(y)\,dx
\nonumber \\
&&
{}+\frac1{4\pi a^3}\Big(\frac{3\lfl}2-\frac\gamma a\Big)\int_{\O}\left(\int_{\partial B_a}\big(\nabla_x v(x) y\cdot   \vece +v_1(x,y)\cdot \vece  \big)\,d\H(y)\right)^2\!\! dx
\nonumber \\
&&
{}-(1-|B_a|\big)\into f(x)\cdot v(x)\,dx. \label{eq.limv}
\end{eqnarray}

On the other hand, recalling \eqref{Vei-var} and  making use of \eqref{eq.unfol-ex}--\eqref{eq.unfol-ex2} and of the definition \eqref{eq.unfol} of $\hat w^\e$, we have that for $\omega\subset\subset\O$ and $\e$ small enough,
\begin{eqnarray}
{\mathcal  F}^\ep(u^\ep)& \ge &
 \frac12 \int_{\omega}\int_{Y\setminus \ol B_a}Q(\bE_y \hwe(x,y))\,dy\,dx
 \nonumber \\
 &&
{}+ \frac\gamma2 \int_{\omega}\int_{\partial B_a}|\nabla_{\tau,y}(P^2_a  (\hwe(x,y)\cdot\vece(y)))|^2\,d\H(y)\, dx
\nonumber \\
&& {}-
 \frac\gamma{a^2}\int_{\omega}\int_{\partial B_a}|P^2_a  (\hwe(x,y)\cdot \vece(y))|^2\,d\H(y)\,dx
 \nonumber \\
 &&
 {} + \frac1{4\pi  a^3}
 \Big(\frac{3\lfl}2-\frac\gamma a\Big)\int_{\omega}\left(\int_{\partial B_a}\hwe(x,y)\cdot\vece(y)\,d\H(y)\right)^2\!\!dx-\int_{\Oe} f\cdot \ue\, dx
 \nonumber \\
 & =: & \G(\hwe)-\int_{\Oe} f\cdot \ue\, dx,
 \label{eq.>liminf}
\end{eqnarray}
where we also used that for a.e.\ $x$
\begin{multline}
\frac\gamma2 \int_{\partial B_a}|\nabla_{\tau,y}(P^2_a  (\hwe(x,\cdot)\cdot\vece))|^2\,d\H-
 \frac\gamma{a^2}\int_{\partial B_a}|P^2_a  (\hwe(x,\cdot)\cdot \vece)|^2\,d\H
 \\
 + \frac1{4\pi  a^3}
 \Big(\frac{3\lfl}2-\frac\gamma a\Big)\left(\int_{\partial B_a}\hwe(x,y)\cdot\vece(y)\,d\H(y)\right)^2\geq0
\end{multline}
by \eqref{ecdt3-0} with $\e=1$, and \eqref{eq.cond-min}.

Now, the inequality above and the positive definite character of $Q$ imply that the quadratic functional $\G$ defined in \eqref{eq.>liminf} is non-negative, hence convex on the space of functions $\X_\omega$ with 
\begin{equation}\label{eq.def-X}
\begin{array}{rcl}
\X_\omega &:= & \big\{ w\in L^2(\omega; H^1(Y\setminus\ol B_a;\R^3)): \ P^2_a  (w\cdot\vece)\in L^2(\omega; H^1(\p B_a))\big\}
\smallskip \\
& = & \big\{ w\in L^2(\omega; H^1(Y\setminus\ol B_a;\R^3)): \ w\cdot\vece\in L^2(\omega; H^1(\p B_a))\big\} .
\end{array}
\end{equation}
Thus, in view of convergences \eqref{eq.lim.uxy}--\eqref{eq.lim.uxy2} and of \eqref{eq.def-u_1}, weak lower semicontinuity yields
\begin{eqnarray}
\lefteqn{\liminf_\ep {\mathcal  F}^\ep(u^\ep)}
\nonumber \\
& \geq & \frac12 \int_{\omega} \int_{Y\setminus B_a}Q(\bE_xu(x)+\bE_y\hat u_1(x,y))\, dy\, dx
\nonumber \\
&& {} + \frac{\gamma}{2}\int_{\omega} \int_{\partial B_a} |\nabla_{\tau,y}P^2_a   (\nabla_x u(x)\, y \cdot   \vece(y) + \hat u_1(x,y)\cdot \vece(y))|^2\,d\H(y)\, dx
\nonumber \\
&& {}
- \frac\gamma{a^2}\int_{\omega}\int_{\partial B_a} |P^2_a   (\nabla_x u(x)\, y\cdot   \vece(y) + \hat u_1(x,y)\cdot \vece(y))|^2\,d\H(y)\, dx
\nonumber \\
&& {} + \frac1{4\pi  a^3}
 \Big(\frac{3\lfl}2-\frac\gamma a\Big)
 \int_{\omega}\left(\int_{\partial B_a} (\nabla_x u(x)y\cdot\vece(y) +\hat u_1(x,y)\cdot \vece(y) )\, d\H(y)\right)^2\!\!dx
 \nonumber \\
&& {}-
\limsup_\e \int_{\Om^\ep}f\cdot u^\ep\,dx. \label{eq.terFev6}
\end{eqnarray}
Now, for any $\omega_\eta\subset\subset\O$ with $|\O\setminus\omega_\eta|\le \eta$, we may write
$$
\int_{\Oe\cap\,\omega_\eta}f\cdot\ue\, dx=\int_{\Oe\cap\,\omega_\eta}f\cdot R^\e\ue\, dx=\int_{\omega_\eta}\chi_{Y\setminus \ol B_a}(x/\e)f\cdot R^\e\ue\, dx.
$$
Therefore, in view of \eqref{eq.conv=Rue} and Rellich's Theorem we have
$$
\int_{\Oe\cap\,\omega_\eta}f\cdot\ue\, dx \longrightarrow (1-|B_a|)\int_{\omega_\eta} fu\,dx.
$$
Since
 $$\left|\int_{\Oe\setminus\omega_\eta}f\cdot\ue dx\right|\le C\|f\|_{L_2(\O\setminus \omega_\eta)}\longrightarrow 0
 $$
 as $\eta\to0$, we deduce that
 \beq\label{eq.limfu}
 \lim_\e\int_{\Oe}f\cdot\ue \,dx=(1-|B_a|)\int_{\O} fu\,dx.
 \end{equation}

By \eqref{eq.limfu} and by letting $\omega\nearrow\O$ in \eqref{eq.terFev6}, we conclude  that $\hat u_1\in L^2(\O;H^1(Y\setminus\ol B_a;\R^3))$ and $P^2_a  (\hat u_1\cdot\vece)\in L^2(\O; H^1(\p B_a;\R^3))$ (and not only locally as in  \eqref{eq.def-u_1} and as implied by \eqref{eq.lim.uxy}) and that
\begin{eqnarray}
\lefteqn{\liminf_\ep {\mathcal  F}^\ep(u^\ep)}
\nonumber \\
& \geq & \frac12 \int_{\Omega} \int_{Y\setminus B_a}Q(\bE_xu(x)+ \bE_y\hat u_1(x,y))\, dy\, dx
\nonumber \\
&& {} + \frac{\gamma}{2}\int_{\Omega} \int_{\partial B_a} |\nabla_{\tau,y}P^2_a   (\nabla_x u(x)\, y \cdot   \vece + \hat u_1(x,y)\cdot \vece(y))|^2\,d\H(y)\, dx
\nonumber \\
&& {}
- \frac\gamma{a^2}\int_{\Omega}\int_{\partial B_a} |P^2_a   (\nabla_x u(x)\, y\cdot   \vece + \hat u_1(x,y)\cdot \vece(y))|^2\,d\H(y)\, dx
\nonumber \\
&& {} + \frac1{4\pi  a^3}
 \Big(\frac{3\lfl}2-\frac\gamma a\Big)
 \int_{\Omega}\left(\int_{\partial B_a} (\nabla_x u(x)y\cdot\vece(y) + \hat u_1(x,y)\cdot \vece(y) )\, d\H(y)\right)^2\!\!dx
 \nonumber \\
&& {}- (1-|B_a|)\int_{\O}f(x)\cdot u(x)\,dx. \label{eq.terFev7}
\end{eqnarray}

By \cite[Lemmas~A.1 and~A.2]{ADMLP} and the assumptions on $\Dir$
the set $\{v\in C^\infty(\ol\Om;\R^3): v=0 \mbox{ on }\Dir\}$
is dense in $H^1_\Dir(\O;\R^3)$.
Further, $C^1_c(\Om;C^1_\sharp(Y\setminus \ol B_a;\R^3))$ is, of course, dense in  $L^2(\O; H^1_\sharp(Y\setminus\ol B_a;\R^3))$ but also in $\X_\Omega$ defined in \eqref{eq.def-X} because $C^1_\sharp(Y\setminus \ol B_a;\R^3)$ is dense in $\X$ defined in \eqref{eq.defXcell}.
 Indeed, take $\psi\in \X$  and $\psi_n\in C^\infty_\sharp(Y\setminus \ol B_a;\R^3)$,  $g_n\in C^\infty(\p B_a)$ converging strongly to $\psi$ and $\psi\cdot\vece$ in $H^1_\sharp(Y\setminus\ol B_a;\R^3)$ and $H^1(\p B_a)$, respectively. Solve, with periodic boundary conditions on $\p Y$,
\begin{equation*}
\begin{cases}
\Delta v_n=\Delta (\psi_n\cdot\vece) \mbox{ in } Y\setminus \ol B_a,\\[2mm] v_n=g_n \mbox{ on }\p B_a,\end{cases}
\end{equation*} 
so that $v_n\in C^\infty_\sharp(Y\setminus\ol B_a)$ converges to $\psi\cdot\vece$  strongly in $H^1(Y\setminus \ol B_a)$ {and ${v_n}|_{\p B_a}$} converges to $\psi\cdot\vece$  strongly in $H^1(\p B_a)$. 
{
Let $\eta>0$ be such that $a+\eta<1/2$ and let $\varphi\in C^\infty_c(B_{a+\eta})$ be a cut-off function such that $\varphi=1$ on $B_{a+\eta/2}$. Define
$$
\zeta_n:=\varphi (v_n- \psi_n\cdot\vece) \vece+\psi_n.
$$
Clearly,  $\zeta_n\in C^\infty_\sharp(Y\setminus\ol B_a;\R^3)$ for every $n$, $\zeta_n\to \psi$ strongly in $H^1_\sharp(Y\setminus\ol B_a;\R^3)$, and $\zeta_n\cdot\vece\to \psi\cdot\vece$ strongly in $H^1(\p B_a)$.}

In view of \eqref{eq.uemi1}, \eqref{eq.limv}, and \eqref{eq.terFev7},
these density results establish that $(u,\hat u_1)$ is a solution of the problem
\begin{multline}\label{eq.pbmili}
\min\left\{\into \F_{per}(\bE_x v(x),v_1)\, dx-(1-|B_a|)\into f\cdot v\,dx:\right.\\[2mm] \left. (v,v_1)\in H^1_\Dir(\Om;\R^3)\times \X_\Omega\right\}
\end{multline}
with $\F_{per}$ defined in \eqref{eq.cell-pb} {and $\X_\Omega$ in \eqref{eq.def-X}}.

Further the minimizing pair $(u,\hat u_1)$ is unique by Remark~\ref{rk.eqform}, which implies the uniqueness of $\nabla_x u \, y+\hat u_1 \in  L^2(\Om;H^1_\sharp (Y\setminus \ol B_a;\R^3))$, hence of  $(u,\hat u_1)$ in $H^1_\Dir(\Om;\R^3)\times L^2(\Om;H^1_\sharp (Y\setminus \ol B_a;\R^3))$. In particular,
convergence of $(\hwe)$ holds along the whole sequence and not only along a suitable subsequence.

Moreover, by taking the $\limsup$ instead of the $\liminf$ in \eqref{eq.terFev6}, we actually get from \eqref{eq.limv}, \eqref{eq.>liminf}, together with the already mentioned density argument that, for any $\omega\subset\subset\O$, as $\e\to 0$,
\begin{equation}\label{eq.conv-xy1}
\bE_y \hwe\to \bE_x u+\bE_y \hat u_1 \mbox{ strongly in }L^2(\omega;L^2(Y\setminus\ol B_a;\R^3))
\end{equation}
 and
\begin{equation}\label{eq.conv-xy2}
\nabla_{\tau,y}(P^2_a  (\hwe(x,\cdot)\cdot\vece)) \to \nabla_{\tau,y}(P^2_a  (\nabla_x u\, y\cdot   \vece +\hat u_1\cdot \vece)) \mbox{ strongly in }L^2(\omega;L^2(\p B_a;\R^3)).
\end{equation}

Note that, in view of  \eqref{eq.pbmili} and of the definition of $\chiF$ as {the} minimizer of \eqref{eq.cell-pb}
for any $F\in \msym$, 
\begin{equation}\label{eq.exp-u1}
\hat u_1(x,\cdot)=\sum_{i,j=1}^{3}(\bE_x u(x))_{ij} \chiij,
\end{equation}
 where $\chiij$ is defined in \eqref{eq.def-chiij}.  

Then remark that  \eqref{eq.pbmili} also reads as
$$
\min\left\{\frac12 \into \Aah \bE v(x)\cdot \bE v(x)\, dx-(1-|B_a|)\into f\cdot v\,dx: \ v\in H^1_\Dir(\Om;\R^3)\right\}
$$
with $\Aah$ defined in \eqref{eq.def-Ah}, which, together with \eqref{eq.exp-u1}, delivers \eqref{eq.hom-pb}.

Finally,    convergences \eqref{eq.conv-xy1} {and} \eqref{eq.conv-xy2} above can in turn be rewritten in terms of $\ue$ as follows. First, for $\omega\subset\!\subset \Omega$, consider $I^\omega_\e\subset I_\e$ {the set of indices $i$ such that $Y^i_\e\subset \omega$} and set $\tilde\omega^\e:=\cup_{i\in I^\omega_\e}(Y^i_\e\setminus  \bBie)$. Then,
{\begin{eqnarray}
\lefteqn{\int_{\tilde\omega^\e}\left|\bE_x\ue(x)-\frac1{\e^3}\int_{Y^{\kappa(x/\e)}_{\e }}\left(\bE_x u(z)+\bE_y\hat u_1(z, x/\e)\right)dz\right|^2 dx}
\nonumber
\\ 
& = &
\sum_{i\in I^\omega_\e}\int_{Y^i_\e\setminus \bBie}\left| \frac1{\e^3}\int_{Y^i_\e}\left(\bE_x  \ue(x)- \bE_x u(z)-\bE_y\hat u_1(z,x/\e)\right)dz     \right|^2  dx
\nonumber \\
& \le & \frac1{\e^3}\sum_{i\in I^\omega_\e}\int_{Y^i_\e\setminus \bBie}\int_{Y^i_\e}\left|\bE_x  \ue(x)- \bE_x u(z)-\bE_y\hat u_1(z,x/\e)\right|^2dz       dx
\nonumber
\\ 
& \le & \int_{\omega}\int_{Y\setminus\ol B_a}\left|1/\e \bE_y\hue(x,y)-\bE_x u(x)-\bE_y\hat u_1(x,y)\right|^2 dydx
\nonumber
\\
& = &
\int_{\omega}\int_{Y\setminus\ol B_a}\left| \bE_y\hwe(x,y)- \bE_x u(x)-\bE_y\hat u_1(x,y)\right|^2 dydx   \stackrel{\e}{\to} 0
\label{eq.conv-cor-0}
\end{eqnarray}
}
where we have used  \eqref{eq.unfol} and \eqref{eq.conv-xy1}.

Since 
$$\bE_x u- \frac1{\e^3}\int_{Y^{\kappa(x/\e)}_{\e }}\bE_x u(z)dz\stackrel{\e}{\longrightarrow} 0\quad \mbox{ strongly in }L^2(\Omega;\mthree),$$
 we conclude from \eqref{eq.conv-cor-0} that 
$$
\int_{\tilde\omega^\e}\left|\bE_x\ue(x)-\bE_x u(x)-\frac1{\e^3}\int_{Y^{\kappa(x/\e)}_{\e }}\bE_y\hat u_1(z, x/\e)dz\right|^2 dx \stackrel{\e}{\longrightarrow} 0,
$$
hence, in view of \eqref{eq.exp-u1}, that 
\begin{equation}\label{eq.strong1}
\int_{\tilde\omega^\e}\left|\bE_x\ue- \bE_x u- \frac1{\e^3}\sum_{i,j=1}^{3}\left(\int_{Y^{\kappa(x/\e)}_{\e}}(\bE_x u)_{ij}(z)\,dz\right) \bE_y\chiij(x/\e)\right|^2 dx  \stackrel{\e}{\longrightarrow} 0.
\end{equation}

Similarly, in the notation of the statement of Theorem \ref{thm.hom}, one can obtain 
{\begin{multline}\label{eq.strong2}
\e\sum_{i\in I^\omega_\e}\int_{\partial \Bie}\Bigg|\nabla_\tau (P^2_{i,\e a}( u^\ep\cdot \vece)) \\[2mm]
\left.-\frac1{\e^3}\nabla_\tau \Bigg(P^2_{a}\Bigg(\int_{Y^i_\e}\Big(\nabla_xu(z) \, y
+\sum_{j,k=1}^{3}(\bE_x u)_{jk}(z)\chi_{jk}(y)\Big)dz\cdot \vece\Bigg)\Bigg)\!\Big|_{y=x/\e}\right|^2d\H\\ \stackrel{\e}{\longrightarrow} 0.
\end{multline}
Equations \eqref{eq.strong1} and  \eqref{eq.strong2} yield \eqref{eq.resuco} and \eqref{eq.resuco2}. Indeed, we can replace $\tilde\omega^\e$ and $I^\omega_\e$ by $\omega\cap\oe$ and $I_\e$  in \eqref{eq.strong1} and \eqref{eq.strong2}, respectively, upon choosing, for any $\omega \subset\!\subset\Omega$,  another set $\omega'\subset\!\subset\Omega$ containing all $Y^i_\e$'s that intersect $\omega$.}

In view of \eqref{eq.conv=Rue},  the proof of Theorem~\ref{thm.hom} is complete.
\end{proof}

\subsection{Elastic enhancement.}\label{sec.enhance}

In this last subsection, we compare the homogenized behavior obtained in the first subsection with that of the {\color{black} elastomer without the inclusions}. We  establish in a specialized setting  the following counterintuitive result: a large enough surface tension will produce elastic enhancement (stronger elasticities) in spite of the lack of resistance to shear in the fluid inclusions.

To that effect we propose to  compare between $1/2 \Aa F\cdot F$ (the elastic energy  for a given constant strain $F$ associated with  the original material occupying the whole volume) to  $1/2 \Aah F\cdot F=\F_{per}(F, \chiF)$.
This will be done through a ``dualization" process.

First we recall the expression \eqref{eq.cell-pb} for $\F_{per}$ as well as  \eqref{ecdt3-0} in  Appendix~A.
We obtain the following inequality for every {$F\in\msym$ and $\psi\in \X$:}
 \begin{multline*}
\F_{per}(F,\psi)\ge \frac12 \int_{Y\setminus\ol B_a}  Q(\bE \psi+F)\, dy
+\frac\gamma3\int_{\p B_a}|\nabla_\tau P^2_a  ((\psi+Fy)\cdot\vece)|^2\, d\H
\\
+\frac{1}{2|B_a|}\left(\lfl-\frac{2\gamma}{3a}\right) \left(\int_{\p B_a}  (\psi+Fy)\cdot\vece\,d\H\right)^2.
\end{multline*}
Taking a supremum over  {triplets} $(\sigma,\xi,t)\in L^2(Y\setminus\ol B_a;\msym)\times L^2(\p B_a;\R^3)\times\R$ with
\begin{equation}\label{eq.cond-sigma}
\begin{cases}\dive\sigma=0 \mbox{ in } Y\setminus\ol B_a,\\[2mm]
\sigma\nu \mbox{ anti-periodic on }\p Y,\quad \sigma\vece\parallel\vece \mbox{ on }\p B_a,\end{cases}
\end{equation}
quadratic duality,  integration by parts and  \eqref{eq.der-area} imply
that
\begin{multline*}
\F_{per}(F,\psi)\ge\sup_{\sigma,\xi,t}\left\{\left(\int_{Y\setminus\ol B_a}\sigma\, dy\right)\cdot F+\int_{Y\setminus\ol B_a} \sigma\cdot \bE \psi\, dy 
\right.\\ 
+\frac{2\gamma}3
\int_{\p B_a}\xi\cdot\nabla_\tau P^2_a  ((\psi+Fy)\cdot\vece)\, d\H+\frac{1}{|B_a|}\left(\lfl-\frac{2\gamma}{3a}\right) t\left(\int_{\p B_a}  (\psi+Fy)\cdot\vece\,d\H\right)\\\left.-\frac12
\int_{Y\setminus\ol B_a}  Q^{-1}(\sigma)\, dy-
\frac\gamma3 \int_{\p B_a}|\xi|^2\, d\H-\frac{1}{2|B_a|}\left(\lfl-\frac{2\gamma}{3a}\right)t^2\right\}
\\
={\sup_{\sigma,\xi,t}}\left\{\left(\int_{Y\setminus\ol B_a}\sigma\, dy+\int_{\p  B_a}(\sigma\vece)\otimes y\, d\H\right)\cdot F-\int_{\p B_a}(\sigma\vece\cdot\vece) ((\psi+Fy)\cdot\vece)\,d\H\right.\\ \left.
+\frac{2\gamma}{3}
\int_{\p B_a}\left(-\dive_\tau\xi+\frac2a\xi\cdot\vece\right) P^2_a  ((\psi+Fy)\cdot\vece)\, d\H\right.\\
+\frac{1}{|B_a|}\left(\lfl-\frac{2\gamma}{3a}\right) t\left(\int_{\p B_a}  (\psi+Fy)\cdot\vece\,d\H\right)-\frac12
\int_{Y\setminus\ol B_a}  Q^{-1}(\sigma)\, dy\\\left.
-\frac\gamma3 \int_{\p B_a}|\xi|^2\, d\H-\frac{1}{2|B_a|}\left(\lfl-\frac{2\gamma}{3a}\right)t^2\right\}.
\end{multline*}
In the right handside of the equality above, the term  $\int_{\p B_a}\dive_\tau\xi\; P^2_a((\psi+Fy)\cdot\vece)\,d\H$ should be understood as a duality product between $H^1(\partial B_a)$ and its dual.

Taking the infimum in $\psi$ in the previous inequality and using that $\inf_\psi\sup_{\sigma,\xi,t}\ge\sup_{\sigma,\xi,t}\inf_\psi$ yields
\begin{multline}\label{eq.ineq-bd1}
\frac12 \Aah F\cdot F\ge\sup_{\sigma,\xi,t}\inf_\psi\left\{\left(\int_{Y\setminus\ol B_a}\sigma\, dy+\int_{\p  B_a}(\sigma\vece)\otimes y\, d\H\right)\cdot F\right.\\ \left.
-\int_{\p B_a}\!\!(\sigma\vece\cdot\vece)((\psi+Fy)\cdot\vece)\, d\H+
\frac{2\gamma}{3}
\int_{\p B_a}\!\!\left(\!\!-\dive_\tau\xi+\frac2a\xi\cdot\vece\right) P^2_a  ((\psi+Fy)\cdot\vece)\, d\H\right.\\
+\frac{1}{|B_a|}\left(\lfl-\frac{2\gamma}{3a}\right) t\left(\int_{\p B_a}  (\psi+Fy)\cdot\vece\,d\H\right)-\frac12
\int_{Y\setminus\ol B_a}  Q^{-1}(\sigma)\, dy\\\left.
-\frac\gamma3 \int_{\p B_a}|\xi|^2\, d\H-\frac{1}{2|B_a|}\left(\lfl-\frac{2\gamma}{3a}\right)t^2\right\}.
\end{multline}
{Given $(\sigma,\psi,t)$ the infimum at the right hand-side in \eqref{eq.ineq-bd1} is $-\infty$ unless the part that is linear in $\psi$ vanishes.
Therefore, the supremum can be restricted to those $(\sigma,\psi,t)$ for which this linear term is zero.
This is, in particular, the case} if $\xi$ and $t$ are such that
\begin{equation}\label{eq.cond-psi-tau}
\begin{cases}\ds
\int_{\p B_a} (-\dive_\tau\xi+\frac2a\xi\cdot\vece ) y\, d\H=0,\\[3mm]\ds
\frac{2\gamma}3( -\dive_\tau\xi+\frac2a\xi\cdot\vece) -\sigma\vece\cdot\vece+\frac{1}{|B_a|}\left(\lfl-\frac{2\gamma}{3a}\right) t=0 \mbox{ on }\p B_a.
\end{cases}
\end{equation}

{We do not know  how to optimally exploit \eqref{eq.ineq-bd1} with the restrictions \eqref{eq.cond-sigma}, \eqref{eq.cond-psi-tau} on $(\sigma,\psi,t)$ as a possible way to demonstrate enhancement for general $\Aa$'s or $F$'s. We propose instead to illustrate enhancement  in the specific case of an isotropic elastomer, i.e.,
$
\Aa_{ijkh}:= \lambda \delta_{ij} \delta_{kh}+\mu (\delta_{ik} \delta_{jh}+\delta_{ih} \delta_{jk}),
$
where $\lambda,\mu$ stand for the Lam\'e constants, and for an axisymmetric shear   strain $F=F(f)$
with
$$
F(f):=-\frac{f}2(\vec{e}_1\otimes\vec{e}_1+\vec{e}_2\otimes\vec{e}_2)+f\vec{e}_3\otimes\vec{e}_3.
$$
Then $1/2\Aa F(f)\cdot F(f)$ is $3/2\mu f^2$. Furthermore we will do so in the dilute limit, that is when
$
a\searrow 0
$.

\vskip1cm

We thus restrict $\sigma$, $\xi$, and $t$ to be of the form
\begin{align}
&\sigma(y)=
\begin{cases}
\sigma_{ij}(y)\vec{e}_i\otimes\vec{e}_j & \mbox{ in }S_{b}=\{y: a<|y|<b\}, \vspace{0.15cm}\\
\overline{\sigma}=\overline{\sigma}_{11}\left(\vec{e}_1\otimes\vec{e}_1+\vec{e}_2\otimes\vec{e}_2\right)+\overline{\sigma}_{33}\vec{e}_3\otimes\vec{e}_3 & \mbox{ in }Y\setminus \overline{S}_{b},
\end{cases} \nonumber\\
&\xi(y)=\beta_7\left(-\frac{y_1 y_3^2}{a^3}\vec{e}_1-\frac{y_2 y_3^2}{a^3}\vec{e}_2+\left(\frac{y_3}{a}-\frac{y_3^3}{a^3}\right)\vec{e}_3\right),\nonumber\\
&t=\beta_8, \label{S-xi-t}
\end{align}
with components
\begin{align}
\sigma_{11}(y)&=\alpha_1+\alpha_2 y_1^2+\alpha_3 y_3^2+\alpha_4 y_1^2 y_3^2,& \sigma_{12}(y)&=\sigma_{21}(y)=\alpha_2 y_1 y_2+\alpha_4 y_1 y_2 y_3^2, \nonumber\\
\sigma_{13}(y)&=\sigma_{31}(y)=\alpha_5 y_1 y_3+\alpha_4 y_1 y_3^3,& \sigma_{22}(y)&=\alpha_1+\alpha_2 y_2^2+\alpha_3 y_3^2+\alpha_4 y_2^2 y_3^2,\nonumber\\
\sigma_{23}(y)&=\sigma_{32}(y)=\alpha_5 y_2 y_3+\alpha_4 y_2 y_3^3,& \sigma_{33}(y)&=\alpha_6+\alpha_7 y_3^2+\alpha_4 y_3^4,\label{sij}
\end{align}
where $\alpha_{1}$ through $\alpha_{7}$ (which are functions of $|y|$) and $\beta_{7}$ and $\beta_{8}$ (which are constants) are spelled out in Appendix B, and where  $\overline{\sigma}_{11}$, $\overline{\sigma}_{33}$ are two arbitrary constants.  It can be checked that the fields $\sigma, \xi, t$ defined above  satisfy {\eqref{eq.cond-sigma} and \eqref{eq.cond-psi-tau}.}

The first equality in \eqref{S-xi-t} corresponds to the stress field in a spherical shell of inner radius $a$ and outer radius $b$ made of an elastic material with elasticity $\Aa$   containing a liquid  with bulk modulus $\lfl$; the solid/liquid interface $r=a$  is endowed with a surface tension $\gamma$. The outer boundary $r=b$ is subject to the affine traction $\ol\sigma\vece$. The choices  for $\xi$ and $t$ in \eqref{S-xi-t} correspond to the traction fields at the interface of the liquid inclusion with the spherical shell in the same problem. Note that the resulting displacement field on the outer boundary of the spherical {shell} is of the form $F(\bar f) y$ for some $\bar f$, which motivates our choice of $\sigma$.

The second equality in \eqref{S-xi-t} corresponds to an affine extension of the stress field in the complement of $B_{b}$ in the unit cell $Y$.

Rather cumbersome but straightforward calculations ensue. First, we    use  \eqref{S-xi-t} in \eqref{eq.ineq-bd1}, so that the terms that are linear in $\psi+Fy$ cancel out. The result is a concave polynomial of degree two in  $\overline{\sigma}_{11}$ and $\overline{\sigma}_{33}$. We compute its maximum in $\overline{\sigma}_{11}$ and $\overline{\sigma}_{33}$. Next we  go to the dilute  case, letting $\theta:=4\pi a^3/3$  tend to $0$. We then obtain the following fully explicit bound:
\begin{multline}\label{eq.ineq-bd3}
\frac12 \Aah F\cdot F\geq \frac{3\mu}{2}f^2\left(1+
\underbrace{\frac{15\mu(\lambda+2\mu) (\gamma/(2\mu a)-1)}{14\mu +9 \lambda+(34 \mu+15\lambda)\gamma/(2\mu a)}}\ \theta\right)+O(\theta^2).\\
(*)\hskip5.28cm
\end{multline}
If $\gamma/\mu>2a$, expression $(*)$ in \eqref{eq.ineq-bd3} will be positive. We conclude that the following holds true.
\begin{proposition}\label{prop.enhance} Consider an isotropic elastomer with Lam\'e coefficients $\lambda, \mu$. If $\gamma/\mu >2a$, then, in the dilute limit $\theta\searrow 0$  ($\theta$ being the volume fraction of the fluid filled cavities),  enhancement will occur for axisymmetric shear strains of the form $F=-f/2(\vec{e}_1\otimes\vec{e}_1+\vec{e}_2\otimes\vec{e}_2)+f\vec{e}_3\otimes\vec{e}_3$, that is,
\begin{equation*}
\frac12\Aah F\cdot F>\frac12\Aa F\cdot F=\frac{3\mu}{2}f^2.
\end{equation*}
\end{proposition}
 The presence of liquid inclusions leads to a stiffer elasticity than that of the elastomer, in spite of the fact that the inclusions have zero shear resistance.

\begin{remark} 
A similar computation could be performed for uniaxial strains of the form $f\vec e\otimes\vec  e$ with $|\vec e\,|=1$.
In that case enhancement can also be achieved but at the expense of choosing both $\gamma$ large with respect to $\mu$  {\it and} $\lfl$  much larger than $\gamma$.

For large $\gamma$'s and still larger $\lfl$, we expect enhancement for general elasticities and strains, but the technicalities involved in deriving a useful lower bound for the homogenized energy remain intractable at present.
\hfill\P\end{remark}

\begin{remark} In the case of  rigid inclusions the homogenized tensor is given by
 \begin{multline*}
 \frac12\Aa^{\rm rigid}F\cdot F\\[2mm] 
 =\min\left\{\frac12 \int_{Y\setminus\ol B_a}  Q(\bE \psi+F)\, dx: \ \psi \in H^1_\sharp(Y\setminus\ol B_a;\R^3),\ \psi=-Fy \mbox{ on }\p B_a\right\}.
 \end{multline*}
 In view of the formula \eqref{eq.def-Ah} for $\Aah$ in the present setting,
$$ 
\frac12 \Aa^{\rm rigid}F\cdot F > \frac12\Aah F\cdot F,
 $$
independently of the values of $\gamma$ and $\lfl$.

Thus, while surface tension on many small liquid inclusions can surprisingly enhance elasticity, it  cannot compete with rigid inclusion, as expected.
\hfill\P\end{remark}

}

\section*{Acknowledgements}

\noindent Support for this work by the National Science Foundation through the Grant DMREF--1922371 is gratefully acknowledged by GAF and OLP. JCD acknowledges support by the project  PID2020-116809GB-I00
of the  Ministerio de Ciencia e Innovaci\'on.
MGM acknowledges support from MIUR--PRIN 2017. MGM is a member of GNAMPA--INdAM.

\setcounter{equation}{0}
\setcounter{section}{1}

\renewcommand{\thesection}{\Alph{section}}

\section*{Appendix A}

In this appendix we review some variants of Poincar\'e's inequality on the boundary of the unit sphere, that are instrumental in the proofs of our main results.

Consider the spectral decomposition of $H^1(\partial B_1)$ associated with the eigenvalues $\mu_\ell$, $\ell\geq 0$, of the Laplace-Beltrami operator on $\partial B_1$, i.e., the solutions in $H^1(\partial B_1)$ of
$$
\int_{\partial B_1} \nabla_\tau \varphi\cdot\nabla_\tau \psi\,d\H=\mu_\ell\int_{\partial B_1} \varphi\psi\,d\H \quad \text{ for every } \psi\in H^1(\partial B_1).
$$
It is well known that $\mu_\ell=\ell(\ell+1)$, $\ell\geq0$ and that the space $V^\ell$ of eigenvectors relative to $\mu_\ell$ has dimension $2\ell+1$ (see \cite[Proposition~4.5]{EF}).
We are especially interested in the space $V^0$, which is given by constant functions, and in the space $V^1$, which is given by the restrictions to $\partial B_1$ of the linear functions in $\RR^3$.

Denote by $P^0$, $P^1$ the orthogonal projections in $L^2(\partial B_1)$ onto the spaces $V^0$ and $V^1$, respectively.
Setting
\begin{equation}\label{proj}
P^2:=I-P^0-P^1, \quad V^2:= (V^0+V^1)^\bot,
\end{equation}
 and recalling that $\mu_1=2$, $\mu_2=6$, we have
$$
\|\varphi\|_{L^2(\p B_1)}^2=\|P^0\varphi\|^2_{L^2(\p B_1)}+\|P^1\varphi\|^2_{L^2(\p B_1)}+\|P^2 \varphi\|^2_{L^2(\p B_1)},
$$
$$
\|\nabla_\tau\varphi\|_{L^2(\p B_1)}^2=\|\nabla_\tau P^1\varphi\|^2_{L^2(\p B_1)}+\|\nabla_\tau P^2 \varphi\|^2_{L^2(\p B_1)},
$$
and
\begin{eqnarray}\label{PWC}
\int_{\partial B_1}|\nabla_\tau P^1\varphi|^2\,d\H & = & 2\int_{\partial B_1} |P^1 \varphi|^2\,d\H,
\\\label{PWC2}
\int_{\partial B_1}|\nabla_\tau P^2 \varphi|^2\,d\H & \geq & 6\int_{\partial B_1} |P^2 \varphi|^2\,d\H
\end{eqnarray}
for every $\varphi\in H^1(\partial B_1)$.
Combining the previous results, we have
\begin{eqnarray*}
\int_{\p B_1}|\varphi|^2\, d\H & \leq & \|P^0\varphi\|^2_{L^2(\p B_1)} + \frac12 \int_{\partial B_1}|\nabla_\tau P^1\varphi|^2\,d\H +\frac16
\int_{\partial B_1}|\nabla_\tau P^2 \varphi|^2\,d\H
\\
& \leq & \frac1{4\pi}\left(\int_{\partial B_1}\varphi\, d\H\right)^2 +\frac12 \int_{\partial B_1}|\nabla_\tau \varphi|^2\, d\H ,
\end{eqnarray*}
where we used that
$$
P^0(\varphi)={\frac1 4\pi}\int_{\partial B_1}\varphi\,d\H.
$$
A simple scaling argument shows that
\begin{equation}\label{PW-sphere}
\int_{\partial B_r}|\varphi|^2\, d\H \leq \frac{r^2}2\int_{\partial B_r}|\nabla_\tau \varphi|^2\, d\H
+ \frac1{4\pi r^2}\left(\int_{\partial B_r}\varphi\, d\H\right)^2
\end{equation}
for every $\varphi\in H^1(\partial B_r)$ and $r>0$.

In Section \ref{sec.hom} we need a refinement {\color{black} of} \eqref{PW-sphere}  established in what follows.

Using the change of variables $x=\e i+ \e a y$, which transforms $\partial B_1$ into $\partial \Bie$,
we deduce that  $H^1(\partial \Bie)$ decomposes as the orthogonal sum of
$V_{i,\e a}^\ell=\{\psi((x-\e i)/\e a):\ \psi\in V^\ell\}.$
Denoting by
\beq\label{proj-e}P^\ell_{i,\e a} \mbox{  the orthogonal projection of }L^2(\partial \Bie)
\mbox{ onto  }V_{i,\e a}^\ell,\eeq
 we have, as above,
$$
\|\varphi\|_{L^2(\p \Bie)}^2=\|P^0_{i,\e a} \varphi\|^2_{L^2(\p \Bie)}+\|P^1_{i,\e a} \varphi\|^2_{L^2(\p \Bie)}+\|P^2_{i,\e a} \varphi\|^2_{L^2(\p \Bie)},
$$
$$
\|\nabla_\tau \varphi\|_{L^2(\p \Bie)^3}^2=\|\nabla_\tau P^1_{i,\e a}\varphi\|^2_{L^2(\p \Bie)}+\|\nabla_\tau P^2_{i,\e a} \varphi\|^2_{L^2(\p \Bie)},
$$
and
\begin{equation}\label{ecdt3}
 \e^2 a^2\int_{\partial \Bie}|\nabla_\tau P^1_{i,\e a} \varphi |^2\,d\H  =  2\int_{\partial \Bie} |P^1_{i,\e a} u|^2\, d\H
 \end{equation}

\begin{equation}\label{ecdt3-0}
\e^2 a^2\int_{\partial \Bie}|\nabla_\tau P^2_{i,\e a} \varphi |^2\, d\H  \geq  6\int_{\partial \Bie} |P^2_{i,\e a} \varphi|^2\, d\H
\end{equation}
for every $\varphi\in H^1(\p\Bie)$. Moreover,
\begin{equation}\label{ecdt4}
P^0_{i,\e a} u={\frac 1 {4\pi \e^2 a^2}}\int_{\partial \Bie}u\,d\H.
\end{equation}
For $K>0$  let us define
$$
\mathcal V^i_{\e,K}(\varphi):=\frac{a^2 \ep^2}{2}\int_{\partial \Bie}|\nabla_\tau \varphi|^2\,d\H+\frac{K}{\ep^2}\left(\int_{\partial \Bie}\varphi\,d\H\right)^2-\int_{\partial \Bie}|\varphi|^2\,d\H
$$
for every $\varphi\in H^1(\partial \Bie)$.
By \eqref{ecdt3} and \eqref{ecdt4} we deduce that
\begin{eqnarray}
\mathcal V^i_{\e,K}(\varphi) & = & \frac{a^2 \ep^2}{2}\int_{\partial \Bie}|\nabla_\tau P^1_{i,\e a} \varphi |^2\,d\H
+\frac{a^2 \ep^2}{2}\int_{\partial \Bie}|\nabla_\tau P^2_{i,\e a} \varphi |^2\,d\H
\nonumber \\
&& {}+ \frac{K}{\ep^2}\left(\int_{\partial \Bie}\varphi\,d\H\right)^2
-\frac1{4\pi a^2\ep^2}\left(\int_{\partial \Bie}\varphi \, d\H\right)^2
\nonumber \\
&& {}
-\int_{\partial \Bie}|P^1_{i,\e a} \varphi|^2\,d\H-\int_{\partial \Bie}|P^2_{i,\e a} \varphi|^2\,d\H
\nonumber\\
& = &  \frac{a^2 \ep^2}{2}\int_{\partial \Bie}|\nabla_\tau P^2_{i,\e a} \varphi |^2\,d\H
+\frac1{\ep^2}\Big(K - \frac1{4\pi a^2}\Big)\left(\int_{\partial \Bie} \varphi \,d\H\right)^2
\nonumber \\
&& {}-\int_{\partial \Bie}|P^2_{i,\e a} \varphi|^2 \, d\H,
\label{identity}
\end{eqnarray}
hence by \eqref{ecdt3-0}
$$
\mathcal V^i_{\e,K}(\varphi) \geq \frac{a^2\ep^2}{3}\int_{\partial \Bie}|\nabla_\tau P^2_{i,\e a} \varphi|^2\,d\H
+\frac1{\ep^2}\Big(K- \frac1{4\pi a^2}\Big)\left(\int_{\partial \Bie}\varphi\,d\H\right)^2.
$$

Thus, in our setting (see \eqref{Vei-var-0}), the following coercivity estimate holds upon taking $K={\lfl a^2/ (2\gamma |B_a|)}$,
\begin{multline}\label{coer}
\frac{\gamma\e}2\int_{\p \Bie}|\nabla_\tau (v\cdot\vece)|^2\, d\H
-\frac{\gamma}{\e a^2}\int_{\p \Bie}  (v\cdot\vece)^2 \, d\H
+\frac{\lfl }{2\e^3|B_a|} \left(\int_{\p \Bie}  v\cdot\vece\,d\H\right)^2\\[2mm]
\ge \frac{\gamma \e}3 \int_{\p \Bie}|\nabla_\tau (P^2_{i,\e a}(v\cdot\vece))|^2\, d\H+ \frac1{4\pi a^3\e^3}\left(\frac32\lfl- \frac{\gamma}{a}\right)\left(\int_{\p \Bie}  v\cdot\vece\,d\H\right)^2.
\end{multline}

\section*{Appendix B}

The functions $\alpha_{1}$ through $\alpha_{7}$ in the components (\ref{sij}) and the constants $\beta_7$ and $\beta_8$ in relations (\ref{S-xi-t}) were computed using Wolfram Mathematica software. They read as
\begin{align*}
\alpha_1=&\frac{15 \beta_1 \mu \lambda |y|^2}{\mu+\lambda}-2 \beta_2 \mu+\beta_5 (2 \mu+3 \lambda)+\frac{2 \beta_6 \mu-\frac{10 \beta_3 \mu^2}{\mu+\lambda}}{|y|^3}+\frac{3 \beta_4 \mu}{|y|^5},\\
\alpha_2=&-\frac{12 \beta_1 \mu \lambda}{\mu+\lambda}+\frac{3 \mu \left(\frac{2 \beta_3 (5 \mu+3 \lambda)}{\mu+\lambda}-2 \beta_6\right)}{|y|^5}-\frac{15 \beta_4 \mu}{|y|^7},\\
\alpha_3=&-\frac{3 \beta_1 \mu (14 \mu+25 \lambda)}{\mu+\lambda}+\frac{18 \beta_3 \mu^2}{|y|^5 (\mu+\lambda)}-\frac{15 \beta_4 \mu}{|y|^7},\\
\alpha_4=&\frac{105 \beta_4 \mu}{|y|^9}-\frac{90 \beta_3 \mu}{|y|^7},\\
\alpha_5=&\frac{6 \beta_1 \mu \lambda}{\mu+\lambda}+\frac{\frac{6 \beta_3 \mu (5 \mu+6 \lambda)}{\mu+\lambda}-6 \beta_6 \mu}{|y|^5}-\frac{45 \beta_4 \mu}{|y|^7},\\
\alpha_6=&\frac{3 \beta_1 \mu |y|^2 (14 \mu+15 \lambda)}{\mu+\lambda}+4 \beta_2 \mu+\beta_5 (2 \mu+3 \lambda)+\frac{\frac{2 \beta_3 \mu^2}{\mu+\lambda}+2 \beta_6 \mu}{|y|^3}+\frac{9 \beta_4 \mu}{|y|^5},\\
\alpha_7=&-\frac{3 \beta_1 \mu (14 \mu+17 \lambda)}{\mu+\lambda}-\frac{6 \mu (\beta_6 (\mu+\lambda)-\beta_3 (8 \mu+9 \lambda))}{|y|^5 (\mu+\lambda)}-\frac{90 \beta_4 \mu}{|y|^7},
\end{align*}
and
\begin{align*}
{\beta_7=6 \beta_2+\frac{18 a^7 \beta_1 \lambda+6 a^2 \beta_3 (5 \mu+3 \lambda)-9 \beta_4 (\mu+\lambda)}{a^5 (\mu+\lambda)},\,\;\,\;\beta_8=3 \left(\frac{\beta_6}{a^3}+\beta_5\right)|B_a|}
\end{align*}
with
\begin{align*}
\beta_1=&k_1(\overline{\sigma}_{11}-\overline{\sigma}_{33}),\,\beta_2=k_2(\overline{\sigma}_{11}-\overline{\sigma}_{33}),\,\beta_3=k_3(\overline{\sigma}_{11}-\overline{\sigma}_{33}),
\,\beta_4=k_4(\overline{\sigma}_{11}-\overline{\sigma}_{33}),\\
\beta_5=&k_5(2\overline{\sigma}_{11}+\overline{\sigma}_{33}),\,\beta_6=k_6(2\overline{\sigma}_{11}+\overline{\sigma}_{33}),
\end{align*}
where $k_1$ through $k_{6}$ are coefficients explicitly known in terms of $\lambda$, $\mu$, $\lfl$, $\gamma$, $a$, and $b$. They read as
\begin{align*}
k_1=&-20 a^3 b^3 K^{-1} (\mu+\lambda) \left[2 a^3 \mu (\mu+\lambda)+a^2 \gamma \mu-2 a b^2 \mu (\mu+\lambda)+b^2 \gamma (\mu+\lambda)\right],\\
k_2=&\frac{1}{6} b^3 K^{-1} \left[50 a^8 \mu \left(28 \mu^2+56 \mu \lambda+27 \lambda^2\right)+200 a^7 \gamma \left(7 \mu^2+11 \mu \lambda+3 \lambda^2\right)-\right.\\
    & 1008 a^6 b^2 \mu (\mu+\lambda)^2-504 a^5 b^2 \gamma \mu
    (\mu+\lambda)-2 a b^7 \mu (14 \mu+9 \lambda) (14 \mu+19 \lambda)-\\
    &\left.b^7 \gamma (34 \mu+15 \lambda) (14 \mu+19 \lambda)\right],\\
k_3=&\frac{5}{6} a^3 b^3 K^{-1} (\mu+\lambda) \left[2 a^8 \mu (14 \mu+19 \lambda)-8 a^7 \gamma (7 \mu+5 \lambda)-2 a b^7 \mu (14 \mu+19 \lambda)+\right.\\
    &\left. b^7 \gamma (14 \mu+19 \lambda)\right],\\
k_4=&a^5 b^5 K^{-1} \left[2 a^6 \mu (\mu+\lambda) (14 \mu+19 \lambda)-8 a^5 \gamma (\mu+\lambda) (7 \mu+5 \lambda)-\right.\\
    &\left. 2 a b^5 \mu (\mu+\lambda) (14 \mu+19 \lambda)-b^5 \gamma \mu (14 \mu+19
   \lambda)\right],
\end{align*}
\begin{align*}
k_5=&\frac{b^3 (2 \gamma-4 a \mu-3 a \lfl)}{12 a^4 \mu (2 \mu+3 \lambda-3 \lfl)+24 a^3 \gamma \mu-3 a b^3 (2 \mu+3 \lambda) (4 \mu+3 \lfl)+6 b^3 \gamma (2 \mu+3 \lambda)},\\
k_6=&\frac{-a^3 b^3 [2\gamma+a (2 \mu+3 \lambda-3 \lfl)]}{3 \left(4 a^4 \mu (2 \mu+3 \lambda-3 \lfl)+8 a^3 \gamma \mu-a b^3 (2 \mu+3 \lambda) (4 \mu+3 \lfl)+2 b^3 \gamma (2 \mu+3
   \lambda)\right)},\\
\end{align*}
where
\begin{align*}
K=&\mu \left[2 a^{11} \mu (14 \mu+9 \lambda) (14 \mu+19 \lambda)-8 a^{10} \gamma (7 \mu+5 \lambda) (14 \mu+9 \lambda)-\right.\\
  &50 a^8 b^3 \mu \left(28 \mu^2+56 \mu \lambda+27 \lambda^2\right)-200 a^7 b^3 \gamma
   \left(7 \mu^2+11 \mu \lambda+3 \lambda^2\right)+\\
   &2016 a^6 b^5 \mu (\mu+\lambda)^2+1008 a^5 b^5 \gamma \mu (\mu+\lambda)-50 a^4 b^7 \mu \left(28 \mu^2+56 \mu \lambda+27 \lambda^2\right)+\\
   & 25 a^3
   b^7 \gamma \left(28 \mu^2+56 \mu \lambda+27 \lambda^2\right)+2 a b^{10} \mu (14 \mu+9 \lambda) (14 \mu+19 \lambda)+\\
   &\left. b^{10} \gamma (34 \mu+15 \lambda) (14 \mu+19 \lambda)\right].
\end{align*}


\begin{thebibliography}{99}

\bibitem{ADMLP} R.~Alicandro, G.~Dal Maso, G.~Lazzaroni and M.~Palombaro.
Derivation of a linearised elasticity model from singularly perturbed multiwell energy functionals.
{\it Arch. Rat. Mech. Anal.} {\bf 230}, 1--45 (2018).

\bibitem{AFP}
L.~Ambrosio, N.~ Fusco, N. and  D.~ Pallara.
{\it Functions of Bounded Variation and Free Discontinuity Problems}, Oxford University Press, Oxford  (2000).

\bibitem{ADH}T.~ Arbogast, J.~Douglas, U.~Hornung.
Derivation of the double porosity model of single phase flow via homogenization theory.
{\it  SIAM J. Math. Anal.} {\bf 21},  823--836 (1990).

\bibitem{Cas1} J. Casado-D\'{\i}az.
Two-scale convergence for nonlinear Dirichlet problems in perforated domains.
{\it   Proc. Roy. Soc. Edinburgh} {\bf 130}-A,  249--276 (2000).

\bibitem{CDG} D. Cioranescu, A. Damlamian, G. Griso.
{\it The  Periodic Unfolding Method}, Springer, Singapore, Series in Contemporary Mathematics, Vol. 3  (2019).

\bibitem{CRMSP}
A.~Chicco-Ruiz, P.~Morin and M.~Sebastian Pauletti.
The shape derivative of the Gauss curvature.
{\it Rev. Un. Mat. Argentina\/} {\bf 59},  311--337 (2018).

\bibitem{DMNP}
G.~Dal Maso, M.~Negri and D.~Percivale.
Linearized elasticity as $\Gamma$-limit of finite elasticity. {\it Set-Valued Anal.\/} {\bf 10}, 165--183  (2002).

\bibitem{EF}
C.~Efthimiou, C.~Frye.
{\it Spherical Harmonics in $p$ Dimensions}. World Scientific Publishing, Singapore, 2014.

\bibitem{EG}
L.C. Evans, R.F. Gariepy.
{\it Measure Theory and Fine Properties of Functions}. Studies in Advanced Mathematics, CRC Press, Boca Raton, 1992.

%\bibitem{FMT}
%G.~A.~Francfort, F.~Murat and L.~Tartar.
%Fourth-order moments of nonnegative measures on $S^2$ and applications. {\it Arch. Rat. Mech. Anal.} {\bf 131}, 305--333 (1995).

\bibitem{FJM}
G.~Friesecke, R.~James and S.~M\"uller.
A theorem on geometric rigidity and the derivation of nonlinear plate theory from three-dimensional elasticity. {\it Com. Pure  Appl. Math.}  {\bf LV}, 1461--1506 (2002).

\bibitem{GLP}
K. Ghosh, O. Lopez-Pamies.
Elastomers filled with liquid inclusions: Theory, numerical implementation, and some basic results. {\it J. Mech Phys. Solids} {\bf 166}, 104930 (2022).

\bibitem{GLLP}
K. Ghosh, V. Lef\`evre and O. Lopez-Pamies.
Homogenization of elastomers filled with liquid inclusions: The small-deformation limit. {\it Journal of Elasticity}, published online (2023).


\bibitem{Gibbs1928}
J. W. Gibbs.
\textit{The Collected Works of J. W. Gibbs}, Vol. 1, Section III (1928).

\bibitem{Gurtin75a}
M. E. Gurtin, A. I. Murdoch.
A continuum theory of elastic material surfaces. \textit{Arch. Ration. Mech. Anal.} \textbf{57}, 291--323.

\bibitem{Gurtin75b}
M. E. Gurtin, A. I. Murdoch.
Addenda to our paper a continuum theory of elastic material surfaces. \textit{Arch. Ration. Mech. Anal.} \textbf{59}, 1--2 (1975).

\bibitem{H}
E. Hebey.
{\it Sobolev  Spaces on Riemannian Manifolds}. Lecture Notes in Mathematics, Springer Berlin, Heidelberg, 1996.

\bibitem{Laplace1806}
P. S. Laplace.
Trait\'e de M\'ecanique C\'eleste, Volume 4, Suppl\'emeent au dixi\`eme livre du Trait\'e de M\'ecanique C\'eleste, pp. 1--79 (1806).

\bibitem{LDLP17}
V. Lef\`evre, K. Danas, and O. Lopez-Pamies.
A general result for the magnetoelastic response of isotropic suspensions of iron and ferrofluid particles in rubber, with applications to spherical and cylindrical specimens. \textit{J. Mech. and Physics of Solids} {\bf 107}, 343--364  (2017).

\bibitem{MM}
C.~Maor, M.G.~Mora.
Reference configurations versus optimal rotations: a derivation of linear elasticity from finite elasticity for all traction forces.
{\it J. Nonlinear Sci.\/} {\bf 31}, 62 (2021).

\bibitem{MR}
M.G.~Mora, F.~Riva.
Pressure live loads and the variational derivation of linear elasticity. {\it Proc. Royal Soc. Edinburgh A}, online, 1--36.

\bibitem{OSY}
O.A.~Oleinik, A.S.~Shamaev, G.A.~Yosifian.
{\it Mathematical Problems in Elasticity and Homogenization,} North-Holland Publishing Co., Amsterdam, 1992.

\bibitem{PGVC}
P.~Podio-Guidugli, G.~Vergara Caffarelli.
Surface interaction potentials in elasticity. {\it Arch. Rational Mech. Anal.\/} {\bf 109}, 343--383 (1990).

\bibitem{Syleetal15}
R.W. Style, R. Boltyanskiy, A. Benjamin, K.E. Jensen, H.P. Foote, J.S. Wettlaufer and E.R. Dufresne.
Stiffening solids with liquid inclusions. \textit{Nature Physics} {\bf 11}, 82--87 (2015).

\bibitem{Young1805}
T. Young. III An essay on the cohesion of fluids. \textit{Phil. Trans. R. Soc.} {\bf 95}, 9565–9587 (1805).

\bibitem{Yunetal19}
G. Yun, S.Y. Tang, S. Sun, D. Yuan, Q. Zhao, L. Deng, S. Yan, H. Du and M.D. Dickey and W. Li.
Liquid metal-filled magnetorheological elastomer with positive piezoconductivity. \textit{Nature Communications} {\bf 10}, 1300  (2019).


\end{thebibliography}
\end{document}